\documentclass{aic}

\aicAUTHORdetails{%
  title = {A Tight Erd\H{o}s-P\'osa Function for Planar Minors}, 
  author = {Wouter Cames van Batenburg, Tony Huynh, Gwena\"el Joret, and Jean-Florent Raymond},
  plaintextauthor = {Wouter Cames van Batenburg, Tony Huynh, Gwenael Joret, Jean-Florent Raymond},
  plaintexttitle = {A Tight Erdős-Posa Function for Planar Minors}, 
}   

\aicEDITORdetails{%
   year={2019},
   number={2},
   received={15 November 2018},   
   revised={19 April 2019},    
   published={30 October 2019},  
   doi={10.19086/aic.10807},      
}   

\usepackage{xcolor}
\usepackage[numbers,sort&compress]{natbib}
\usepackage{enumitem}

\usepackage{tikz}
\tikzstyle{black node} = [draw, circle, fill = black, minimum size = 4pt, inner sep = 0pt]
\tikzstyle{normal} = [draw=none, fill = none, rectangle, minimum size =0]

 \usepackage{aliascnt}

 \newtheorem{theorem}{Theorem}[section]

 \newaliascnt{lemma}{theorem}
 \newtheorem{lemma}[lemma]{Lemma}
 \aliascntresetthe{lemma}

 \newaliascnt{observation}{theorem}
 
 \aliascntresetthe{observation}

 \newaliascnt{corollary}{theorem}
 \newtheorem{corollary}[corollary]{Corollary}
 \aliascntresetthe{corollary}

 \newaliascnt{conjecture}{theorem}
 
 \aliascntresetthe{conjecture}

 \newaliascnt{claim}{theorem}
 \newtheorem{claim}[claim]{Claim}
 \aliascntresetthe{claim}

\def\cqedsymbol{\ifmmode$\lrcorner$\else{\unskip\nobreak\hfil
\penalty50\hskip1em\null\nobreak\hfil$\lrcorner$
\parfillskip=0pt\finalhyphendemerits=0\endgraf}\fi}

\newcommand{\cqed}{\renewcommand{\qed}{\cqedsymbol}}

\newcommand{\arxiv}[1]{\href{http://arxiv.org/abs/#1}{arXiv:#1}}

\newcommand{\R}{\mathbb{R}}
\newcommand{\N}{\mathbb{N}}

\newcommand{\sR}{\mathsf{R}} 
\newcommand{\mR}{\mathcal{R}} 
\newcommand{\varphip}{\varphi}
\newcommand{\varphipa}{\varphi'}
\newcommand{\Gsmall}{G_{\mathsf{s}}}
\newcommand{\Gbig}{G_{\mathsf{b}}}
\newcommand{\opt}{\mathsf{OPT}}

\newcommand{\ep}{Erd\H{o}s-P\'osa}
\newcommand{\sref}[1]{\scriptscriptstyle \ref{#1}}

\DeclareMathOperator{\tw}{\mathsf{tw}}
\newcommand{\mc}{\mathcal}

\begin{document}
\begin{frontmatter}[classification=text]

\title{A Tight Erd\H{o}s-P\'osa Function\\ for Planar Minors\titlefootnote{A preliminary version of this paper appeared as an extended abstract in the Proceedings of the Thirtieth Annual ACM-SIAM Symposium on Discrete Algorithms (SODA '19)~\cite{SODAversion}.}} 

\author[wcvb]{Wouter Cames van Batenburg\thanks{Supported by an ARC grant from the Wallonia-Brussels Federation of Belgium.}}
\author[th]{Tony Huynh\thanks{Supported by ERC Consolidator Grant 615640-ForEFront.}}
\author[gj]{Gwena\"el Joret\thanks{Supported by an ARC grant from the Wallonia-Brussels Federation of Belgium.}}
\author[jfr]{Jean-Florent Raymond\thanks{Supported by ERC Consolidator Grant 648527-DISTRUCT.}}

\begin{abstract}
Let $H$ be a planar graph.
By a classical result of Robertson and Seymour, there is a function $f:\N \to \R$ such that for all $k \in \N$ and all graphs $G$, either $G$ contains $k$ vertex-disjoint subgraphs each containing $H$ as a minor,  or there is a subset $X$ of at most $f(k)$ vertices such that $G-X$ has no $H$-minor.
We prove that this remains true with $f(k) = c k \log k$ for some constant $c=c(H)$.
This bound is best possible, up to the value of $c$, and improves upon a recent result of Chekuri and Chuzhoy [STOC 2013], who established this with $f(k) = c k \log^d k$ for some universal constant $d$.
The proof is constructive and yields a polynomial-time $O(\log \opt)$-approximation algorithm for packing subgraphs containing an $H$-minor.
\end{abstract}
\end{frontmatter}

\section{Introduction}

In 1965, Erd\H{o}s and P\'osa~\cite{Erdos1965independent} proved that there is a function $f(k)=O(k\log k)$ such that for every graph $G$ and every $k \in \N$, either $G$ contains $k$ vertex-disjoint cycles, or there is a set $X$ of at most $f(k)$ vertices such that $G-X$ is a forest.
Many variants and generalizations of this theorem have been developed over the years, such as for cycles satisfying various constraints~\cite{KKM11, BJS18, KK13, Thomassen1988, Joos2017, BBR07, FH14, BL81, KW06, PW12, RR01, Tho01, Wol11, KK12, HJW18}, directed cycles~\cite{RRST96, HM13}, matroid circuits~\cite{GK09}, and immersions~\cite{liu2015, GKRT16}; see~\cite{Reed97tree, Raymond2017} for surveys.

In this paper, the objects of interest are graph minor models.
A graph $H$ is a \emph{minor} of a graph $G$ if $H$ can be obtained from a subgraph of $G$ by contracting  edges.
If $H$ is not a minor of $G$, then $G$ is said to be \emph{$H$-minor free}.
For every graph $H$,
an {\em $H$-model} $\mc{M}$ in a graph $G$ is a collection
$\{M_{x} \subseteq G: x\in V(H)\}$ of vertex-disjoint connected subgraphs of $G$
such that $M_{x}$ and $M_{y}$ are linked by an edge in $G$ for every edge $xy \in E(H)$. 
We define $V(\mc{M})=\bigcup_{x\in V(H)}V(M_x)$ and the \emph{size} of $\mc{M}$ as $|V(\mc{M})|$.

Two $H$-models $\mc{M}$ and $\mc{M'}$ are \emph{disjoint} if $V(\mc{M}) \cap V(\mc{M'}) = \emptyset$.  It is easy to see that $H$ is a minor of $G$ if and only if there is an $H$-model in $G$.

Let $\nu_{H}(G)$ be the maximum number of pairwise disjoint $H$-models in $G$.
Let $\tau_{H}(G)$ be the minimum size of a subset $X \subseteq V(G)$ such that $G - X$
has no $H$-model.
Clearly, $\nu_{H}(G) \leq \tau_{H}(G)$.  We say that the \emph{\ep{} property holds for $H$-models} if there exists a \emph{bounding function} $f \colon \N \to \mathbb{R}$ such that
\[
\tau_H (G) \leq f(\nu_H(G))
\]
holds for every graph $G$.

By a classical result of Robertson and Seymour~\cite{robertson1986graph}, the \ep{} property holds for $H$-models if and only if $H$ is planar;
the fact that it does hold when $H$ is planar is a consequence of their Grid Minor Theorem.
The original bounding function $f$ obtained by Robertson and Seymour for planar $H$ was exponential.
In 2013, Chekuri and Chuzhoy~\cite{chekuri2013large} proved that one can take $f(k)= c k \log^d (k+1)$ as bounding function for some universal constant $d$, and some constant $c=c(H)$.

No explicit value for the constant $d$ is given in~\cite{chekuri2013large} but a quick analysis of their proof suggests that it is at least a double-digit integer.
Our main result is that we can take $d=1$, which is best possible.

\begin{theorem}[Main theorem]
\label{thm:main}
For each planar graph $H$, there exists a constant $c=c(H)$ such that the \ep{} property holds for $H$-models with bounding function $f(k)=ck \log (k+1)$.
\end{theorem}

A $O(k \log k)$ bounding function is best possible for the following reason.
For $H=K_3$ an $\Omega(k \log k)$ lower bound on bounding functions was already established by Erd\H{o}s and P\'osa~\cite{Erdos1965independent}.
This lower bound holds more generally when $H$ is any planar graph containing a cycle, as can be seen  by considering $n$-vertex graphs $G$ with treewidth $\Omega(n)$ and girth $\Omega(\log n)$ (as constructed in \cite{morgenstern1994existence}, for instance).
Then $\tau_{H}(G)= \Omega(n)$ (because $H$-minor free graphs have treewidth $O(1)$ when $H$ is planar) and $\nu_{H}(G)= O(n/\log n)$ (because each $H$-model contains a cycle).
We also note that if, on the other hand, $H$ is a forest, then the $\Omega(k \log k)$ lower bound does not apply, and it is in fact already known that there is a $O(k)$ bounding function~\cite{FJW2013}.

Before pursuing further, let us emphasize that the constant $c=c(H)$ we obtain in the proof of \autoref{thm:main} is enormous, in fact it is not even known to be computable.
On the other hand, $c$ depends polynomially on $|H|$ in the bounding function $f(k)= c k \log^d (k+1)$ established by Chekuri and Chuzhoy~\cite{chekuri2013large} (this follows from~\cite{chekuri2013large} combined with their celebrated Polynomial Grid Minor Theorem~\cite{chekuri2016polynomial,chuzhoy2015excluded}).
Thus, our main result can be seen as a trade-off, where we decrease the value of $d$ to an optimal $d=1$ at the price of a much bigger constant factor $c$. 
Computability of our constant $c$ in \autoref{thm:main} does not follow from our proofs because we use a result of Fomin, Lokshtanov, Misra, and Saurabh~\cite{FominFdel} about minor-minimal graphs $G$ with $\tau_{H}(G)=t$, whose proof is non-constructive (see Section~\ref{sec:tighep}). 
Finding good bounds on $c$ as a function of $|H|$ is left as an open problem, in particular it would be interesting to determine whether $c$ could depend polynomially on $|H|$. 

Prior to this paper, when $H$ is planar but not a forest, a $O(k \log k)$ bounding function was known to hold for $H$-models if $H$ is a triangle~\cite{Erdos1965independent}, a cycle~\cite{FH14, MNSW17}, a multigraph consisting of two vertices linked by parallel edges~\cite{Chatzidimitriou2017logopt}, and more generally if $H$ is any minor of a wheel~\cite{ErdosPosaWheels}. 
The authors of~\cite{ErdosPosaWheels} developed general tools to tackle arbitrary planar graphs $H$, together with some techniques that are specific to wheels.
In this paper we build on their approach.
Our main technical contribution is a series of lemmas which allowed us to develop the `right' generalization of the objects used in~\cite{ErdosPosaWheels}.
An overview of the proof will be given shortly but first let us mention some combinatorial and algorithmic consequences of our result.

\section{Consequences of our results}
\label{sec:combcor}

We describe in this section several consequences of our results. Their proofs are given in Sections \ref{sec:apalg} and~\ref{sec:procor}.

\subsection*{Approximation algorithms for packing and covering models}
Our proof of \autoref{thm:main} is constructive, in the sense that it can be turned into a polynomial-time algorithm computing both a collection $\mathcal{C}$ of $k$ disjoint $H$-models in the input graph $G$, and a subset $X$ of at most $c' k \log(k+1)$ vertices such that $G-X$ has no $H$-model, for some constant $c'$ depending on the constant $c$ in \autoref{thm:main} and for some $k \in \N$.
Note that $\mathcal{C}$, $X$ together witness the fact that (1) $|X|$ is within a $O(\log \tau_H(G))$ factor of $\tau_H(G)$ (since $k \leq \tau_H(G)$), and (2) $|\mathcal{C}|$ is within a $\Omega\left (\frac{1}{\log \nu_H(G)} \right)$ factor of $\nu_H(G)$ (since $k \leq \nu_H(G) \leq c' k \log(k+1)$).
Thus, we get $O(\log(\opt))$-approximation algorithms for both the packing and covering problems associated to planar $H$-models.

\begin{corollary}
\label{cor:approx}
For each fixed planar graph $H$, there is a polynomial-time $O(\log(\opt))$-approximation algorithm both for computing $\nu_H(G)$ and $\tau_H(G)$.
\end{corollary}

The result for covering is already known.
In fact, for every planar graph $H$, there is even a constant factor approximation algorithm for computing $\tau_H(G)$.
Indeed, a randomized constant factor approximation was first developed by Fomin, Lokshtanov, Misra, and Saurabh~\cite{FominLMS12}, and very recently a deterministic one was obtained by Gupta, Lee, Li, Manurangsi, and W{\l}odarczyk~\cite{Gupta2018LLMWlosing}.

On the other hand, the result for packing is new.
It is also close to best possible in the following sense:
When $H=K_3$ the packing problem corresponds to the well-studied problem of packing cycles, which is known to be quasi-NP-hard to approximate to within a ratio of $O(\log^{\frac{1}{2}-\varepsilon} \opt)$~\cite{Friggstad07}. 
We note also that when $H$ is a forest, $\nu_H(G)$ can be approximated to within a constant factor~\cite{FJW2013}.

\subsection*{Large treewidth graph decompositions} A second consequence of our main theorem is the following partitioning corollary.
\begin{corollary}
\label{cor:tw}
There is a function $s_{\sref{cor:tw}}:\N \to \N$ such that for all integers $r, k \geq 1$, every graph $G$ of treewidth at least
\[
s_{\sref{cor:tw}}(r) \cdot k \log (k+1)
\]
has $k$ vertex-disjoint subgraphs $G_1, \dots, G_k$, each of treewidth at least $r$.
\end{corollary}

In particular, the treewidth of every graph not containing $k$ disjoint copies of a fixed planar graph $H$ as a minor is $O(k \log k)$, where the hidden constant depends on $H$.
This is best possible when $H$ contains a cycle (see the paragraph following \autoref{thm:main}).
A similar result with an $s(r) \cdot k \log^d (k+1)$ bound for some universal constant $d$ was obtained by Chekuri and Chuzhoy \cite[Theorem~1.1]{chekuri2013large}.
Again, we remark that, while the poly-logarithmic dependency on $k$ in their bound is not optimal, their theorem has the extra advantage that $s$ can be taken as a polynomial, which is not the case in our proof of \autoref{cor:tw}.

\subsection*{Computing minor-closed bidimensional parameters.}

Let $\pi$ be a \emph{graph parameter}, that is, a function mapping
graphs to integers and that is constant within each isomorphism
class. We say that $\pi$ is \emph{minor-closed} if $\pi(H) \leq
\pi(G)$ for every minor $H$ of every graph $G$.
In \cite{chekuri2013large}, Chekuri and Chuzhoy gave algorithms to
compute graph parameters satisfying certain conditions.

\begin{theorem}[{\cite[Theorem~5.3]{chekuri2013large}}]\label{th:bidimcc}
  Let $\pi$ be a minor-closed parameter that is positive on all
  graphs with treewidth at least $p$, is at least the sum over the
  components of a disconnected graph, and can be computed in
  time $h(w) n^{O(1)}$ given a tree-decomposition of width $w$ of the
  graph.

  Then there is a constant $d$ and an algorithm that, given an
  $n$-vertex graph $G$ and an integer $k$, decides whether $\pi(G)
  \leq k$ in time
  \[
    \left ( 2^{O(p^2 k \cdot \log^d pk )} + h(O(p^2k\cdot \log^dpk))\right ) n^{O(1)}.
  \]
\end{theorem}

Note that the requirements of \autoref{th:bidimcc} are satisfied by
several well-studied parameters such as feedback vertex set, vertex
cover, and more generally any packing or covering problem of models of
a fixed planar graph (as described at the beginning of the section).
By plugging the improved bounds of our partitionning result
\autoref{cor:tw} in the proof of \autoref{th:bidimcc} from \cite{chekuri2013large}, we obtain
the following result.

\begin{corollary}\label{cor:miclopar1}
  Let $\pi$ be a minor-closed parameter that is positive on all
  graphs with treewidth at least $p$, is at least the sum over the
  components of a disconnected graph, and can be computed in
  time $h(w) n^{O(1)}$ given a tree-decomposition of width $w$ of the
  graph.

  Then there is an algorithm that, given an
  $n$-vertex graph $G$ and an integer $k$, decides whether $\pi(G)
  \leq k$ in time
  \[
    \left ( 2^{O(s_{\sref{cor:tw}}(p) k\log k)} +
      h(O(s_{\sref{cor:tw}}(p) k \log k))\right ) n^{O(1)}.
  \]
\end{corollary}

Observe that \autoref{cor:miclopar1} improves the dependence on $k$ of the algorithm from \autoref{th:bidimcc}, at the cost of a worse
dependence on~$p$. However, in the natural
setting where $\pi$ is fixed and we want to check $\pi(G) \leq k$ for
various pairs $(G,k)$, $p$ is a constant so its contribution is less relevant.
As noted in \cite{chekuri2013large}, the
requirements on $\pi$ can also be stated as follows.

\begin{corollary}\label{cor:miclopar2}
  Let $\pi$ be a minor-closed parameter that is positive on some
  $t$-vertex planar graph $H$, is at least the sum over the
  components of a disconnected graph, and can be computed in
  time $h(w) n^{O(1)}$ given a tree-decomposition of width $w$ of the
  graph.

  Then there is a function $s_{\sref{cor:miclopar2}}$ and an algorithm
  that, given an $n$-vertex graph $G$ and an integer $k$, decides
  whether $\pi(G) \leq k$ in time
  \[
    \left ( 2^{s_{\sref{cor:miclopar2}}(t) k \log k} + h(s_{\sref{cor:miclopar2}}(t) k \log k)\right ) n^{O(1)}.
  \]
\end{corollary}

\subsection*{\texorpdfstring{\ep{}}{Erdös-Pósa} property in minor-closed classes}
For a graph $H$ and a class $\mathcal{G}$ of graphs, we say that the \ep{} property holds for $H$-models \emph{in} $\mathcal{G}$ if there exists a bounding function $f \colon \N \to \mathbb{R}$ such that $\tau_H (G) \leq f(\nu_H(G))$ holds for every graph $G \in \mathcal{G}$.
Restricting the class $\mathcal{G}$ sometimes yields improved bounding functions.
For instance, while the bounding function in the classic \ep{} theorem is $\Theta(k \log k)$, it can be improved to $O(k)$ when restricted to planar graphs~\cite{BIENSTOCK1992163}.
In fact, this is true more generally for $H$-models for any fixed planar graph $H$ when restricted to any proper minor-closed class $\mathcal{G}$, as shown by Fomin, Saurabh, and Thilikos \cite{FominST11stre}.

\begin{theorem}[Fomin, Saurabh, and Thilikos \cite{FominST11stre}]\label{thm:linear}
  Let $\mathcal{G}$ be a proper minor-closed graph class and let $H$ be a planar graph.
  Then there exists a constant $c:=c(\mathcal{G}, H)$ such that the \ep{} property holds for $H$-models in $\mathcal{G}$ with bounding function $f(k)=ck$.
\end{theorem}

As it turns out this theorem also follows directly from our main technical theorem (stated in the next section).

\subsection*{Packing cycles with modularity constraints}

In 1988, Thomassen obtained the following modularity-constrained variant
of the \ep{} theorem:
\begin{theorem}[Thomassen \cite{Thomassen1988}]\label{thm:paritep} For
every $m \in \N$ there is a function $f$ such that, for every $k \in
\N$ and every graph $G$, either $G$ contains $k$ vertex-disjoint
cycles of length 0 modulo $m$, or there is a subset $X$ of at most
$f(k)$ vertices such that $G-X$ has no such cycle.
\end{theorem}

Wollan~\cite{Wol11} obtained a similar statement for cycles
with non-zero length modulo $m$, when $m$ is odd.
As proved by Dejter and Neumann-Lara \cite{dejter1987unboundedness}, the same statement does not hold in general for
cycles of length $m'$ modulo $m$, when $m'\in [m-1]$. Thomassen's
upper-bound $f(k) =2^{2^{O(k)}}$ (for fixed $m$) has later been improved to $f(k) =
O(k \log^d k)$ for some $d$ by Chekuri and Chuzhoy~\cite{chekuri2013large}, who used a partitioning theorem
similar to our \autoref{cor:tw}.
As a consequence of our main theorem, we obtain a $O(k \log k)$ bounding function for cycles of length 0 modulo $m$, which is
the same as in the original \ep{} Theorem.

\begin{corollary}\label{cor:parity} For every positive integer $m$
there is a constant $c:=c(m)$ such that, for every
$k \in \N$ and every graph $G$, either $G$ contains $k$
vertex-disjoint cycles of length 0 modulo $m$, or there is a subset
$X$ of at most $c\cdot k \log(k+1)$ vertices such that $G-X$ has no
such cycle.
\end{corollary}
Extremal graphs showing that this bound is tight (up to the
value of $c$) can be obtained from extremal graphs for the original
\ep{} Theorem by subdividing every edge $m-1$ times. We actually
prove a stronger statement about modularity-constrained subdivisions
of planar subcubic graphs, whose proof we postpone to~\autoref{sec:procor}.

\section{Overview of the proof}\label{sec:ovww}
In this paper, all logarithms are binary.  Unless otherwise specified, 
the graphs we consider are finite, simple, and undirected. 
In particular, when contracting edges of a graph, we subsequently delete resulting loops and parallel edges. Let $G$ be a graph.
We use $|G|$ and $\Vert G \Vert$ as shorthand for $|V(G)|$ and $|E(G)|$, respectively.

A \emph{separation} of a graph $G$ is a pair $(A,B)$ of subsets of
$V(G)$ such that $A\cup B = V(G)$ and $G$ has no edge from $A  \setminus  
B$ to $B  \setminus   A$.
Observe that our definition allows $A$ or $B$ to be empty.
The {\em order} of the separation is $|A\cap B|$.

The heart of our proof is the following technical theorem.

\begin{theorem}[Main technical theorem]
\label{thm:main_technical}
  For every $p \in \N$, every planar graph $H$, and every non-decreasing function $g$ with
  $g(0)=1$, there is a constant $\sigma \in \N$ such that for every graph $G$, at least one of the following holds.
  \begin{enumerate}[label = (\roman*)]
  \item \label{e:smallmodel} $G$ contains an $H$-model of size at most $\sigma$;
  \item \label{e:shallow-clique} $G$ contains a $K_p$-model of size at most $\sigma \log |G|$;
  \item \label{e:sep} $G$ has a separation $(A,B)$ of order at most
    $\sigma$
    such that $G[A]$ does not contain $H$ as a minor and $|A| \geq
    g(|A\cap B|)$.
  \end{enumerate}
\end{theorem}

\autoref{thm:main} follows quickly from \autoref{thm:main_technical} using previous results.
We give the derivation in \autoref{sec:tighep}. Thus, it only remains to prove \autoref{thm:main_technical}.

To give a high-level idea of our proof strategy for \autoref{thm:main_technical}, we sketch it for the case $H = K_3$.
Note that every cycle in our graph $G$ is a $K_3$-model.
First, we consider a maximum-size collection $\mathcal{P}$ of paths of length $\omega$, for some large enough constant $\omega$.
Assume for simplicity that these paths cover all vertices of $G$.
If one of the paths in $\mathcal{P}$ is not induced, we find a cycle of length at most $\omega$.
Similarly, if two of these paths are connected by at least two edges, we get a cycle of length at most $2\omega$.
In both cases, we find a $K_3$-model of size at most $2\omega \leq \sigma$ for a suitable choice of the constant $\sigma$, and \ref{e:smallmodel} is satisfied.
Thus we may assume this does not happen.

Then, we consider the auxiliary graph $G'$ on vertex set $\mathcal{P}$ where two
vertices are adjacent if the corresponding paths are connected by an
edge in $G$.
If $G'$ has large enough minimum degree (as a function of $p$), then a known result (see~\autoref{thm:small-minors} in the next section) yields a $K_p$-model of size $O(\log |G'|)$ in
$G'$, which translates into a $K_p$-model of size $O(\omega \log |G|)$ in~$G$, which is outcome \ref{e:shallow-clique}.

Hence, we may assume that $G'$ has a vertex of degree bounded by some function of $p$.
Then the corresponding path $P \in \mathcal{P}$ has neighbors in only a few other paths of
$\mathcal{P}$.
By letting $A := V(P)$ and letting $B$ be the rest of the graph plus the vertices of $A$ with a neighbor in $G-A$, we obtain outcome \ref{e:sep} (assuming $\omega$ has been chosen large enough).

While the arguments leading to outcomes \ref{e:shallow-clique} and
\ref{e:sep} above work for all planar graphs $H$, this approach fails in general as the existence of many edges between two paths of
$\mathcal{P}$ does not always yield a small $H$-model (outcome \ref{e:smallmodel}).

The aforementioned result of~\cite{ErdosPosaWheels} for the case where $H$ is a wheel avoids this difficulty by packing paths and cycles instead of just paths. However, this technique breaks down when trying to pack subgraphs having a vertex of degree at least $3$.

In our proof, we addressed this difficulty by introducing a family of objects called
\emph{orchards} and considering \emph{orchard packings} as a
counterpart to the family $\mathcal{P}$ of paths/cycles.
Roughly speaking, orchards have the property that two disjoint orchards connected by
many edges either can be combined into more desirable structures (in the same
sense that two paths connected by two edges induce a cycle in the proof sketch above),
or the orchards can be separated in a `clean way' from each other using a small set of vertices.
This allows us to conclude similarly as above. However, the proof is more involved.

The rest of the paper is organized as follows. The next section contains the
general definitions and results we use. In \autoref{sec:tighep} we
prove \autoref{thm:main} assuming \autoref{thm:main_technical}.
Orchards and orchard packings are introduced in \autoref{sec:orchards}
and \autoref{sec:packing_orchards}, along with some key separation lemmas.
Using these results we finally prove
\autoref{thm:main_technical} in \autoref{sec:tech}.
The proofs of the algorithmic and combinatorial consequences of our results stated in \autoref{sec:combcor} are given in Sections \ref{sec:apalg} and \ref{sec:procor}, respectively.

\section{Preliminaries}

A \emph{tree-decomposition} of a graph  $G$ is a tree $T$ together with subsets $B_t$ of $V(G)$ for each $t \in V(T)$ satisfying

\begin{itemize}
\item $V(G)= \bigcup_{t \in V(T)} B_t$,
\item for each $uv \in E(G)$, there exists $t \in V(T)$ such that $u,v \in B_t$, and
\item for each $v \in V(G)$, the set of all $w \in V(T)$ such that $v \in B_w$ induces a subtree of $T$.
\end{itemize}
The \emph{width} of the tree-decomposition is $\max_{ t \in V(T)} \{|B_t|-1 \}$.  The \emph{treewidth} of $G$, denoted $\tw(G)$, is the minimum width taken over all tree-decompositions of $G$.

\begin{theorem}[Robertson and Seymour~\cite{robertson1986graph}]\label{th:gridminor}
There exists a function $f_{\sref{th:gridminor}} \colon \N \to \N$ such that for every $t\in \N$, every graph of treewidth at least $f_{\sref{th:gridminor}}(t)$ contains every $t$-vertex planar graph as a minor.
\end{theorem}

By the results of Chekuri and Chuzhoy~\cite{chekuri2016polynomial, chuzhoy2015excluded}, $f_{\sref{th:gridminor}}$ can be bounded from above by a polynomial function.

We do not directly use tree-decompositions in this paper. Instead, we use the following dual notion. A \textit{bramble} $\mathcal{B}$ in a graph $G$ is a collection of vertex sets of connected subgraphs of $G$, called \textit{bramble sets} of $\mathcal{B}$, such that for all $B,B'\in \mathcal{B}$, $|B \cap B'|\geq 1$ or there is an edge between $B$ and $B'$.
The \textit{order} of $\mathcal{B}$ is the minimum size of a set $W\subseteq V(G)$ such that $W$ intersects all bramble sets.

\begin{theorem}[Seymour and Thomas~\cite{seymour1993graph}]\label{th:brambleduality}
Let $k\geq 0$ be an integer. A graph has treewidth at least $k$ if and only if it contains a bramble of order at least $k+1$.
\end{theorem}

We also require the following two theorems.

\begin{theorem}[Erd\H{o}s-Szekeres Theorem \cite{erdiis1935combinatorial}]\label{thm:es}
Let $p,q \in \N$. Every sequence of at least $(p - 1)(q - 1) + 1$ distinct integers contains an increasing subsequence of length $p$ or a decreasing subsequence of length $q$.
\end{theorem}

\begin{theorem}[Fiorini, Joret, Theis, and Wood~\cite{Fiorini20121226}, see also \cite{Montgomery2015loga, Shapira2015}]\label{thm:small-minors}
There is a function $f_{\sref{thm:small-minors}} \colon \N \to \N$ such that, for every $n,p \in \N$, if an $n$-vertex graph has average degree at least $f_{\sref{thm:small-minors}}(p)$, then it contains a $K_p$-model on $O(\log n)$ vertices.
\end{theorem}

\section{From the main technical theorem to the main theorem}
\label{sec:tighep}

In this section, we show how \autoref{thm:main}
can be deduced from \autoref{thm:main_technical}.
We follow the same line of proof as in \cite{ErdosPosaWheels} by considering a minimal counterexample and showing that the outcomes of \autoref{thm:main_technical} contradict its minimality. By \emph{minor-minimal} we mean minimal with respect to the minor ordering.
We rely on the following results.

\begin{theorem}[Fomin, Lokshtanov, Misra, and Saurabh {\cite[Corollary~1]{FominFdel}}]
\label{thm:fominkernel}
For every planar graph $H$, there is a polynomial $p_{\sref{thm:fominkernel}}$ such that for every $t \in \N$, every graph $G$ with $\tau_H(G) = t$ and minor-minimal with this property satisfies $|G| \leq p_{\sref{thm:fominkernel}}(t)$.
\end{theorem}

Let us emphasize that the polynomial $p_{\sref{thm:fominkernel}}$ in \autoref{thm:fominkernel} depends (non-constructively) on $H$. 

\begin{theorem}[Fiorini, Joret, and Wood \cite{FJW2013}]\label{thm:fjw13}
  For every connected planar graph $H$, there is a computable and non-decreasing function $f_{\sref{thm:fjw13}}\colon \N\to \N$ such that, for every graph $G$, if $(A,B)$ is a separation of $G$ where $G[A]$ is $H$-minor free and $|A| \geq f_{\sref{thm:fjw13}}(|A\cap B|)$, then there exists a graph $G'$ such that
  \[
    \tau_H(G') = \tau_H(G),\qquad \nu_H(G') = \nu_H(G),\qquad \text{and} \qquad |G'| < |G|.
  \]
\end{theorem}
\autoref{thm:fjw13} as originally stated in \cite{FJW2013} does not
guarantee that $f_{\sref{thm:fjw13}}$ is non-decreasing. We can however
easily obtain this property by defining $f_{\sref{thm:fjw13}}(k) = \max_{i\in \{0, \dots, k\}} f(k)$, with $f$ the function given in~\cite{FJW2013}, and clearly $f_{\sref{thm:fjw13}}$ then has the properties claimed in \autoref{thm:fjw13}.

\begin{lemma}[Aboulker, Fiorini, Huynh, Joret, Raymond, and Sau
{\cite[Lemma 2.7]{ErdosPosaWheels}}, reworded]\label{lem:minorep} 
Let $H'$ be a planar graph and let $f'$ be a bounding function for $H'$-models. 
Then, for each minor $H$ of $H'$, there is a bounding function $f$ for $H$-models with $f\in O(f')$. 
\end{lemma}

We are now ready to prove \autoref{thm:main}, assuming \autoref{thm:main_technical}.

\begin{proof}[Proof of \autoref{thm:main}]
  Let us first assume that $H$ is connected.
  We explain at the end of the proof how the result extends to disconnected graphs.

  Let $\alpha$ and $\beta$ be positive integers such that
  for every integer $k\geq 1$, we have
  $p_{\sref{thm:fominkernel}}(k) \leq \alpha
  k^\beta$, where $p_{\sref{thm:fominkernel}}$
  is the function of
  \autoref{thm:fominkernel} for $H$. Such numbers exist as this function is a polynomial.

  Let $f_{\sref{thm:fjw13}}$ be the function of \autoref{thm:fjw13} for the graph
  $H$. Clearly we can assume $f_{\sref{thm:fjw13}}(0) = 1$. Let $\sigma$ be the constant of \autoref{thm:main_technical}
  for $p=|H|$ and for the function~$f_{\sref{thm:fjw13}}$. 
  We prove \autoref{thm:main} for $f(k) = ck \log (k+1)$, where $c$ is a positive integer such that  $c \geq \sigma (\log \alpha + \beta \log c + 2\beta)$.

  Towards a contradiction, suppose
  $\tau_H(G) > f(\nu_H(G))$ for some graph $G$. Among all such graphs, we choose $G$ such that
  the tuple ($\nu_H(G), |G|, \Vert G \Vert)$ is lexicographically minimum.       
  Let $k = \nu_H(G)\geq 1$.

  We apply \autoref{thm:main_technical} on $G$ with $p  =  |H|$ and $g = f_{\sref{thm:fjw13}}$.
  According to \autoref{thm:fjw13}, the outcome \ref{e:sep} of \autoref{thm:main_technical} implies the existence of a graph $G'$ such that
  \[
    \tau_H(G') = \tau_H(G),\qquad \nu_H(G') = \nu_H(G),\qquad \text{and} \qquad |G'| < |G|.
  \]
  This would however contradict the minimality of $G$.
  Therefore we may now assume that one of the first two outcomes of \autoref{thm:main_technical} holds.   
  Which of the two outcomes holds is not important for the rest of the proof, as we will only use the fact that $G$ contains a model $\mathcal{M}$ of $H$ of size at most $\sigma \cdot \log |G|$, which is true in both cases. 
  Using properties of $G$ we will show that $|V(\mathcal{M})| \leq \sigma \cdot \log |G| \leq c \log (k+1)$. 
  Once this is established, using that the graph $G - V(\mathcal{M})$ is not a counterexample to \autoref{thm:main_technical}, 
  we will conclude that $G$ cannot be a counterexample either. 

  The definition of $G$ implies that it is minor-minimal with the property $\tau_H(G) > f(\nu_H(G))$.
  Thus, if $G'$ is a proper minor of $G$, then $\tau_H(G') \leq f(\nu_H(G')) \leq f(\nu_H(G)) = f(k)$ (since $\nu_H(G') \leq \nu_H(G)$).
  In particular, $G$ is minor-minimal with the property $\tau_H(G) > f(k)$.
  Now, observe that $\tau_H(G) \leq \tau_H(G-v) + 1$ for any vertex $v\in V(G)$ (simply add $v$ to an optimal hitting set for $G-v$).
  Hence, $\tau_H(G) \leq \tau_H(G-v) + 1 \leq f(\nu_H(G-v)) + 1 \leq f(\nu_H(G)) + 1 = f(k)+1$ and $\tau_H(G) = f(k) + 1$.
  Therefore, we can apply \autoref{thm:fominkernel}.
  \begin{align*}
    |G| &\leq \alpha \tau_{H}(G)^{\beta}&
                                                      (\text{\autoref{thm:fominkernel}})\\
        &= \alpha(f(k) + 1)^{\beta} \\
        &\leq \alpha(c(k+1)^2)^{\beta}.
  \end{align*}
  Then
  \begin{align*}
     |V(\mathcal{M})| &\leq \sigma \cdot \log |G| & (\text{definition of}\ \mathcal{M})\\
        & \leq \sigma (\log \alpha + \beta \log c + 2 \beta) \log (k+1)\\
        & \leq c \log (k+1) & (\text{definition of}\ c).\\
  \end{align*}
  Let us consider the graph $G'  := G - V(\mathcal{M})$.
  Observe that
  \[
    \nu_H(G) \geq \nu_H(G') + 1 \qquad \text{and} \qquad \tau_H(G)
    \leq \tau_H(G') + c \log (k + 1).
  \]
  By minimality of $G$, we have $\tau_H(G') \leq f(\nu_H(G'))$. Then
  \begin{align*}
    \tau_H(G) &\leq f(\nu_H(G')) + c \log (k + 1) \\
              &\leq f(k-1) + c \log (k + 1) &(f\ \text{is
                                                          non-decreasing})\\
              &\leq f(k).
  \end{align*}
  Therefore, $G$ is not a counterexample, a contradiction.

  We now consider the case where $H$ is not connected. Let $H'$ be a
  planar connected graph with $V(H') = V(H)$ and $E(H') \supseteq
  E(H)$. Such a graph can be obtained from planar
  drawings of the components of $H$ by adding edges between their
  external faces in a planar way. As shown in the first part of the
  proof, there is a bounding function $f'(k)  = c' k \log
  (k+1)$ for $H'$-models, for some constant $c'$ depending on $H'$ only.
  By applying \autoref{lem:minorep} to $H'$, $f'$, and $H$,
  we obtain a bounding function $f$ for $H$-models which is of the same order of magnitude as $f'$, as desired.
\end{proof}

\section{Orchards}\label{sec:orchards}

We prove in this section a series of lemmas about bramble-like objects that we call orchards.
Given positive integers $a, b$, an {\em $a \times b$-orchard} $\mathsf{R}$ in $G$ is a collection $P_1, \dots, P_a$ of $a$ pairwise vertex-disjoint paths, called {\em horizontal paths}, and a collection $T_1, \dots, T_b$ of $b$ pairwise vertex-disjoint trees, called {\em vertical trees}, such that

\begin{itemize}
	\item $P_i \cap T_j$ is non-empty and connected (and thus a path) for each $i\in [a]$ and $j\in [b]$, and
	\item each leaf of $T_j$ is on some horizontal path, for each $j\in [b]$. 	
\end{itemize}
With a slight abuse of notation we also write $\mathsf{R}$ for the subgraph formed by the union of the horizontal paths and vertical trees of $\mathsf{R}$.
It should be clear from the context whether $\mathsf{R}$ means the orchard itself or the corresponding subgraph of $G$.

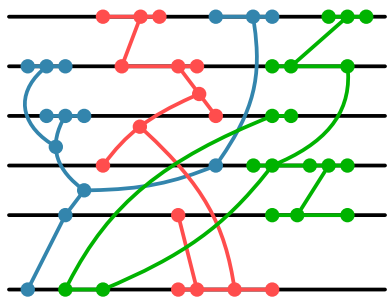
\begin{figure}[h]
  \centering
  \begin{tikzpicture}[every node/.style = {black node, line width = 0}, every path/.style = {line width = 0.05cm, line cap = round}, yscale = 0.66]
    \foreach \i in {1,..., 5}{
      \draw (0, -\i) -- ++(5,0);
    }
    \draw (0, -6.5) -- ++(5,0);
    \draw[color = blue!60!green!80, every node/.style = {black node, fill = blue!60!green!80}]
    (0.25,-2) node (b1) {} -- ++(0.25,0) node (b2) {} -- ++(0.25,0) node (b3) {} 
    (0.5,-3) node (b4) {} -- ++(0.25,0) node (b5) {} -- ++(0.25,0) node (b6) {} 
    (b5) to[bend right] node[pos=0.333] (b7) {} (1,-4.5) node (b8) {} 
    (b7) to[bend left = 40] (b2)
    (0.25, -6.5) node (b10) {} -- ++(0.5,1.5) node (b9) {} -- (b8)
    (b8) to[bend right = 15] (2.75,-4) node (b11) {}
    (2.75,-1) node (b13) {} --++(0.75,0) node (b14) {} node[pos = 0.66] (b12) {} 
    (b12) to[bend left = 15] (b11);
    \draw[color = red!70, every node/.style = {black node, fill =red!70}]
    (1.25,-1) node (r1) {} -- ++(0.5,0) node (r2) {} -- ++(0.25,0) node (r3) {} 
    (r2) --++(-0.25,-1) node (r4) {} -- ++ (0.75,0) node (r5) {} -- ++(0.25,0) node (r6) {} 
    (r5) -- ++(0.5,-1) node (r8) {} node[midway] (r7) {} 
    (r7) to[bend right=15] node[midway] (r9) {} (1.25,-4) node (r10) {} 
    (r9) to[bend left = 15] (3,-6.5) node (r11) {} (r11) -- ++(0.5,0) node (r12) {} (r11) -- ++(-0.5,0) node (r13) {} --++(-0.25,0) node (r14) {} 
    (r13) --++(-0.25, 1.5) node (r15) {};
    \draw[color = green!70!black, every node/.style = {black node, fill = green!70!black}]
    (r3) ++(2.25,0) node (g1) {} -- ++(0.25,0) node (g2) {} --++(0.25,0) node (g3) {} 
    (g2)  --++(-0.75,-1) node (g4) {} --++(-0.25,0) node (g5) {} (g4) -- ++(0.75,0) node (g6) {}
    (g6) to[bend left] ++(-1,-2) node (g7) {} (g7) -- ++(-0.25,0) node (g8) {} (g7) --++(0.5,0) node (g9) {} --++(0.25,0) node (g10) {} --++(0.25,0) node (g11) {} 
    (g10) ++(-0.75,-1) node (g14) {} --++(1,0) node (g13) {} node[pos=0.333] (g12) {} (g12) -- (g10) 
    (g7) to[bend left=15] (1.25, -6.5) node (g15) {} -- ++(-0.5,0) node (g16) {} 
    (g16) to[bend left = 20] (3.5,-3) node (g17) {} (g17) --++(0.25,0) node (g18) {};
  \end{tikzpicture}
  \caption{A $6\times 3$-orchard. Horizontal paths are depicted in black and vertical trees in color.}
  \label{fig:orch}
\end{figure}

Orchards are similar to brambles in the sense that they can serve as certificates for large treewidth.
In fact, every large enough orchard contains a bramble of large order (see the proof of \autoref{lem:big-orchard}).
However, they are more structured, which makes them easier to handle.
We note that grids are particular examples of orchards.
Thus, in this sense orchards lie somewhere in between grids and brambles.
We note that a concept similar to orchards is that of {\em grid-like minors}, introduced by Reed and Wood~\cite{ReedWood11}.
Grid-like minors are collections of paths whose intersection graphs are bipartite and contain a large clique minor.
While orchards and grid-like minors have common features (note that the intersection graph of the horizontal paths and vertical trees of an orchard is a complete bipartite graph), in general they are incomparable objects.

The main result of this section is a separation lemma for orchards, \autoref{lem:separ}, which will be used in the proof of \autoref{thm:main_technical}.

\begin{lemma}\label{lem:big-orchard}
 If a graph $G$ contains an $(f_{\sref{th:gridminor}}(t)+1) \times (f_{\sref{th:gridminor}}(t)+1)$-orchard, then $G$ contains every $t$-vertex planar graph as a minor.
\end{lemma}
\begin{proof}
Let $\mathsf{R}$ be an $(f_{\sref{th:gridminor}}(t)+1) \times (f_{\sref{th:gridminor}}(t)+1)$-orchard with a collection of horizontal paths $\mathcal{P}$ and a collection of vertical trees $\mathcal{T}$.
Consider the bramble
$\mathcal{B}:=\left\{ T \cup P \mid  (T,P) \in \mathcal{T} \times \mathcal{P} \right\}$ in $G$. Since the vertical trees are vertex-disjoint, the horizontal paths are vertex-disjoint, and $|\mathcal{T}|=|\mathcal{P}|=f_{\sref{th:gridminor}}(t)+1$,  it follows that the order of $\mathcal{B}$ is at least $f_{\sref{th:gridminor}}(t)+1$. By \autoref{th:brambleduality}, $G$ has treewidth at least $f_{\sref{th:gridminor}}(t)$, and therefore by \autoref{th:gridminor}, $G$ contains every $t$-vertex planar graph as a minor.
\end{proof}

Let $\mathsf{R}$ be an $a \times b$-orchard with horizontal paths $P_1, \dots, P_a$ and vertical trees $T_1, \dots T_b$.  Let $\mathsf{P}=\bigcup_{i \in [a]} P_i$ and $\mathsf{T}=\bigcup_{i \in [b]} T_i$.  We say that $\mathsf{h}$ is a \emph{horizontal section} if $\mathsf{h}=P_i \cap T_j$ for some $i \in [a], j \in [b]$ or if $\mathsf{h}$ is a component of $\mathsf{P} - V(\mathsf{T})$.  Note that the set of all horizontal sections is a collection of vertex-disjoint paths whose union covers all vertices of $\mathsf{P}$.  Let $W$ be the set of vertices $w$ such that for some $i \in [b]$, $w \in V(T_i)  \setminus  V(\mathsf{P})$ and $\deg_{T_i}(w) \geq 3$.  We say that $\mathsf{v}$ is a \emph{vertical section} if $\mathsf{v}$ is a vertex in $W$ (seen as a single-vertex path) or if $\mathsf{v}$ is a component of $\mathsf{T} - (V(\mathsf{P}) \cup W)$.  We say that $\mathsf{s}$ is a \emph{section} if $\mathsf{s}$ is a horizontal or a vertical section.  Note that the set of all sections is a collection of vertex-disjoint paths whose union covers all vertices of $\mathsf{R}$. In the proofs of Lemmas~\ref{lem:separateorchard} and~\ref{lem:separatepathmix} below, we will use several times that if $\mathsf{R}$ has $a\geq 2$ horizontal paths, then each of its vertical trees defines at most $1+3 \cdot (a-2) < a^2$ vertical sections.\footnote{
\emph{Proof}. We proceed by induction on $a$. If $a=2$, then every vertical tree has one vertical section. Let $a\geq3$. Let $\mathsf{R}$ be an orchard. In the following, whenever we speak of a \emph{neighbor}, it is with respect to $\mathsf{R}$ viewed as a graph. Given a vertical tree $T$, let $P$ be a horizontal path such that exactly one vertex of $V(P)\cap V(T)$ has a neighbor $v_1$ in $V(T) \setminus  V(P)$, and this neighbor is unique. Let $T'$ be the tree obtained from $T-V(P)$ by iteratively deleting the unique leaf that is not on any horizontal path other than $P$. Let $v_1,v_2,\ldots, v_{k-1}$ be the sequence of such leaves, and let $v_k$ denote the neighbor of $v_{k-1}$ in $V(T')$. Let $\mathsf{R}'$ be the orchard obtained from $\mathsf{R}$ by deleting $P$ and replacing $T$ with $T'$. Let $\mathsf{s}$ be the unique section of $\mathsf{R}'$ containing $v_k$. Every vertical section of $T$ in $\mathsf{R}$ which is not a vertical section of $T'$ in $\mathsf{R}'$ must be one of the following. (i) the path $v_1v_2\ldots v_{k-1}$ or (ii) $v_k$  or (iii)  one of the at most two components of $\mathsf{s} -v_k$. 
(We remark that situations (ii) and (iii) only apply if $\mathsf{s}$ is a vertical section of $\mathsf{R}'$ and $\mathsf{s} \neq v_k$.) By induction, there are at most $1+3(a-3)$ vertical sections of $\mathsf{R}'$ on $T'$. By the discussion above, $T$ has at most three more vertical sections.}

We define a \textit{myriapod} $C$ to be a tree of maximum degree at most $3$ such that all its degree $3$ vertices are on a single path $P$, called the \textit{spine} of $C$.
The components of $C-V(P)$ will be called the \textit{legs} of $C$.

We show that the sections of an orchard can be covered by few myriapods.

\begin{lemma}\label{lem:myriapodcover} 
  Let $\sR$ be an $a \times b$-orchard.
  There is a collection $\mathcal{C}$ of at most $a^2$ subgraphs of $\sR$ such that:
  \begin{itemize}
  \item every element of $\mathcal{C}$ is a myriapod whose spine is a
    horizontal path of $\sR$ and each of whose legs is contained in some vertical tree;
  \item every section of $\sR$ is contained in some element of $\mathcal{C}$.
  \end{itemize}
\end{lemma}
\begin{proof}
For each ordered pair of distinct horizontal paths $(P_i,P_j)$ in $\mathsf{R}$, we take $P_i$ and  extend it to a myriapod by adding to it the following legs. For each vertical tree $T$ in $\mathsf{R}$, we add the (unique) subpath $P(T,i,j)$ of $T$ that has endpoints in $P_i$ respectively $P_j$ but has no vertex of $V(P_i)\cup V(P_j)$ in its interior. By the uniqueness of the paths $P(T,i,j)$ and because vertical trees are vertex-disjoint, the resulting graph is a myriapod. There are less than $a^2$ ordered pairs of horizontal paths.
Since each horizontal section of $\mathsf{R}$ is contained in some horizontal path and each vertical section is contained in some connecting subpath $P(T,i,j)$, it follows that the constructed myriapods together cover all sections of $\mathsf{R}$.
\end{proof}

Recall that each vertical tree intersects each horizontal path in a
subpath and that these subpaths are disjoint. Thus, each horizontal
path $P$ defines two symmetric total orders on the vertical trees, which
are given by the order in which we meet these trees when following $P$
from one endpoint to the other.

We say that an $a \times b$-orchard $\mathsf{R}$ is \textit{tame} if its vertical
trees appear in the same order along every horizontal path.
Formally, $\mathsf{R}$ is tame if there is a permutation $\pi$ of
$[b]$ such that for every $i \in [a]$, we meet
the horizontal trees of $\mathsf{R}$ in the order $T_{\pi(1)}, \dots,
T_{\pi(b)}$, or the reverse order, when following $P_i$ from one endpoint to the other.

Given a horizontal section $\mathsf{t}$ of a horizontal path $P$ and a vertical tree $T$ in an orchard $\mathsf{R}$, we say that $\mathsf{t}$ is \emph{bordered} by $T$ if $\mathsf{t}$ does not intersect $T$ and, with respect to $P$ viewed as a graph, one of the endpoints of $\mathsf{t}$ has a neighbor which is a vertex of $T$. If additionally (given an ordering of $P$ `from left to right') there is such a neighbor to the left (right) of $\mathsf{t}$, then we say that $\mathsf{t}$ is bordered by $T$ on its left (right).

An orchard $\mathsf{R'}$ is a \textit{suborchard} of an orchard $\mathsf{R}$ if $\mathsf{R'}$ is obtained from $\mathsf{R}$ by selecting a subset of its horizontal paths and a subset of its vertical trees.

\begin{lemma}\label{lem:tameorchard} 
There exists a function $f_{\sref{lem:tameorchard}}(a,b)$ such that, for every $a, b \geq 1$, if $\mathsf{R}$ is an  $a \times f_{\sref{lem:tameorchard}}(a,b)$-orchard,
then $\mathsf{R}$ contains a tame $a\times b$-suborchard $\mathsf{R}'$.
\end{lemma}
\begin{proof}
We claim that we may take $f_{\sref{lem:tameorchard}}(a,b):=b^{2^{a-1}}$.  The proof is by induction on $a$. Note that every  $1\times f_{\sref{lem:tameorchard}}(1,b)$-orchard is tame, and $f_{\sref{lem:tameorchard}}(1,b)=b$, so the claim holds for $a=1$.

For the inductive step, let $P_1, \dots, P_a$ be the horizontal paths
of $\mathsf{R}$ and let us consider the orchard obtained from $\mathsf{R}$ by ignoring $P_a$ and contracting some edges of the vertical trees so that the leaves of each vertical tree lie on $V(P_1) \cup \dots \cup V(P_{a-1})$. More precisely, from each vertical tree we iteratively delete the leaves that are not in $V(P_1)\cup \ldots \cup V(P_{a-1})$. Observe that $f_{\sref{lem:tameorchard}}(a,b)=f_{\sref{lem:tameorchard}}(a-1, b^2)$. Therefore, by induction, this orchard
contains a tame $(a-1)\times b^2$-suborchard
$\mathsf{R}^-$.

Let $T_1^-, \dots, T_{b^2}^-$ be the vertical trees of $\mathsf{R}^-$,
named according to the order in which they intersect $P_1$. 
Since $\mathsf{R}^-$  is tame, this is also the order in which they
intersect $P_i$ for all $i \in [a-1]$.  Let $T_1, \dots, T_{b^2}$ be
the corresponding trees in $\mathsf{R}$.
Choose one of the two possible orientations for $P_a$ arbitrarily and let $g(1), \dots, g(b^2)$ be
the order in which $T_1, \dots, T_{b^2}$  intersect $P_a$.  By
\autoref{thm:es}, $g(1), \dots, g(b^2)$ contains an increasing
or decreasing subsequence $g'(1), \dots, g'(b)$.  Let $T_1', \dots,
T_b'$ be the vertical trees of $\mathsf{R}$ corresponding to $g'(1),
\dots, g'(b)$.  By reversing the orientation of $P_a$ if necessary, we
obtain a tame $a \times b$ suborchard $\mathsf{R}'$ of $\mathsf{R}$,
as required.
\end{proof}

Using Lemmas~\ref{lem:myriapodcover} and~\ref{lem:tameorchard}, we now derive separation lemmas that will be key tools in the main proof.
These lemmas and those in \autoref{sec:packing_orchards} are all parameterized by some positive integer $m$.
In \autoref{sec:tech} we will apply these lemmas with the value $m = f_{\sref{th:gridminor}}(|H|)+1$.

Given two disjoint subsets $A,B$ of vertices of a graph $G$, we say that $A$ \textit{sees} $B$ if there is an edge in $G$ linking a vertex of $A$ to one of $B$.

\begin{lemma}\label{lem:separateorchard}
Let $m\in \mathbb{N}$. Suppose that $\mathsf{R}$ is an $a \times b$-orchard and $\mathsf{R'}$ is an $a' \times b'$-orchard vertex-disjoint from $\mathsf{R}$ in a graph $G$, with $a, a' \in [m]$.
Then, for each $c \geq 1$ at least one of the following holds.
\begin{itemize}
\item $G[V(\mathsf{R}) \cup V(\sR')]$ contains $2^{m}$ pairwise vertex-disjoint $(a+1) \times c$-orchards;
\item there exists $X \subseteq V(\sR')$ with $|X| \leq f_{\sref{lem:separateorchard}}(c,m):= \left(2 \cdot(2^m \cdot c)^{2^{m}} + 1\right)^2 \cdot m^6$ such that $V(\sR') \setminus X$ sees at most $g_{\sref{lem:separateorchard}}(c,m):= f_{\sref{lem:separateorchard}}(c,m)^2$ sections of the orchard $\mathsf{R}$.
\end{itemize}

\end{lemma}
\begin{proof}
  Let $\mathcal{S}$ denote the set of sections of $\mathsf{R}$.
  Consider the auxiliary bipartite graph $G_{\text{bip}}$ with vertex partition $(V(\mathsf{R'}), \mathcal{S})$ with the vertices of $\mathsf{R'}$ in one part and the sections $\mathcal{S}$ of $\mathsf{R}$ in the other part, such that $v\mathsf{s}$ is an edge in $G_{\text{bip}}$ if and only if $v \in V(\mathsf{R'})$ sees section $\mathsf{s} \in \mathcal{S}$ in~$G$.

Let $X \subseteq V(\mathsf{R'})$ denote the set of vertices of $\mathsf{R'}$ that see more than $f_{\sref{lem:separateorchard}}(c,m)$ sections of $\mathsf{R}$.   Suppose $|X|\geq f_{\sref{lem:separateorchard}}(c,m)$.
Then $G_{\text{bip}}$ contains a matching $M$ of size~$f_{\sref{lem:separateorchard}}(c,m)$.
Suppose on the other hand that $|X| \leq f_{\sref{lem:separateorchard}}(c,m)$.
If $V(\mathsf{R'}) \setminus X$ sees at most $g_{\sref{lem:separateorchard}}(c,m)$ sections of $\mathsf{R}$, then we are done. Thus we may suppose that $V(\mathsf{R'}) \setminus X$ sees more than $g_{\sref{lem:separateorchard}}(c,m)$ sections of $\mathsf{R}$.
By definition of $X$, each vertex in $V(\mathsf{R'}) \setminus X$ sees at most $f_{\sref{lem:separateorchard}}(c,m)$ sections of $\mathsf{R}$.
Hence, there exists a matching of size $g_{\sref{lem:separateorchard}}(c,m) / f_{\sref{lem:separateorchard}}(c,m) = f_{\sref{lem:separateorchard}}(c,m)$ between $V(\mathsf{R'}) \setminus X$ and the sections of~$\mathsf{R}$.

Thus in both cases, $G_{\text{bip}}$ contains a matching $M$ of size~$f_{\sref{lem:separateorchard}}(c,m)$. From this fact we will derive that $G[V(\sR) \cup V(\sR')]$ contains $2^m$ pairwise vertex-disjoint $(a+1) \times c$-orchards.

By \autoref{lem:myriapodcover} applied to $\mathsf{R}$ and by the pigeonhole principle, there is a myriapod $C_{\mathsf{R}}$ in $\mathsf{R}$ such that:
\begin{itemize}
\item the spine of $C_{\mathsf{R}}$ is a horizontal path of $\mathsf{R}$ and each leg of $C_{\mathsf{R}}$ is a subgraph of a vertical tree of $\mathsf{R}$; and
\item at least $\frac{1}{a^2} f_{\sref{lem:separateorchard}}(c,m)$ sections matched by $M$ are contained in $C_{\mathsf{R}}$.
\end{itemize}
If $a \geq 2$, then each leg of $C_{\mathsf{R}}$ is a subgraph of a vertical tree of $\mathsf{R}$ and hence contains at most $1 + 3(a-2) < a^2$ sections.
(If $a=1$, then $C_{\mathsf{R}}$ has no legs.)
It follows that there is a submatching $M_2 \subseteq M$ of size at least $\frac{1}{a^4} \cdot f_{\sref{lem:separateorchard}}(c,m)$ such that the sections of $\mathsf{R}$ matched by $M_2$ are

\begin{enumerate}[label = (\alph*)]\label{enum:matchingsectionsofH}
\item all on the spine of $C_{\mathsf{R}}$, or  \label{enum:all_spine_H}
\item on distinct legs of $C_{\mathsf{R}}$. \label{enum:distinct_legs_H}
\end{enumerate}

By \ref{enum:distinct_legs_H} we mean that each section matched by $M_2$ is on some leg of $C_{\mathsf{R}}$ and no two such sections are on the same leg of $C_{\mathsf{R}}$.

We can apply a similar reduction to the vertices of $\mathsf{R'}$ matched by $M_2$.
By \autoref{lem:myriapodcover}, the vertices of $\mathsf{R'}$ can be covered with at most $a'^2$ (recall $a'$ is the number of horizontal paths of
$\mathsf{R}'$) myriapods whose spines are horizontal paths of $\mathsf{R'}$ and each of whose legs is a subgraph of a vertical tree of $\mathsf{R'}$.
Thus there is such a myriapod $C_{\mathsf{R'}}$ in $\mathsf{R'}$ such that at least $\frac{1}{a'^2} \cdot \frac{1}{a^4} \cdot f_{\sref{lem:separateorchard}}(c,m)$ vertices of $C_{\mathsf{R'}}$ are matched by $M_2$.

Next, we claim that we can find a submatching $M_3\subseteq M_2$ of size at least
\[
\sqrt{\frac{1}{m^6} \cdot f_{\sref{lem:separateorchard}}(c,m)} =  2 \cdot (2^m\cdot c)^{2^{m}} + 1
\]
such that the vertices of $\mathsf{R'}$ matched by $M_3$ are
\begin{enumerate}[label = (\arabic*)]\label{enum:matchingverticesofHprime}
\item all on the spine of $C_{\mathsf{R'}}$, or \label{enum:all_spine_Hp}
\item on distinct legs of $C_{\mathsf{R'}}$, or \label{enum:distinct_legs_Hp}
\item all on a single leg of $C_{\mathsf{R'}}$. \label{enum:single_leg_Hp}
\end{enumerate}
This can be seen as follows.
Let $y:= \frac{1}{m^6} \cdot f_{\sref{lem:separateorchard}}(c,m)$.
As $a,a'\leq m$, the myriapod  $C_{\mathsf{R'}}$ has at least $y$ vertices matched by $M_2$.
A \emph{part} is the spine or a leg of $C_{\mathsf{R'}}$.
If some part of $C_{\mathsf{R'}}$ contains at least $\sqrt{y}$ matched vertices, then \ref{enum:all_spine_Hp} or \ref{enum:single_leg_Hp} holds, and we are done.
Otherwise, strictly more than $y / \sqrt{y} = \sqrt{y}$ parts have at least one matched vertex.
Since $\sqrt{y}$ is an integer, there are at least $\sqrt{y}+1$ such parts.
By possibly discarding the spine, we obtain $\sqrt{y}$ distinct legs each having a matched vertex, and thus \ref{enum:distinct_legs_Hp} holds.

We now extend the $a \times b$-orchard $\mathsf{R}$ to an $(a+1)\times \left( (2^m \cdot c)^{2^{m}} \right)$-orchard, as follows.  As the $(a+1)$-th horizontal path of the new orchard we take (in case~\ref{enum:all_spine_Hp} and~\ref{enum:distinct_legs_Hp}) the spine of $C_{\mathsf{R'}}$ or (in case~\ref{enum:single_leg_Hp}) the leg of $C_{\mathsf{R'}}$ that is matched by $M_3$. For each edge $e=v\mathsf{s}$ in $M_3$ we choose an edge in the original graph $G$, which has endpoints $v\in V(\mathsf{R'})$ and some vertex on section $\mathsf{s}$. In case~\ref{enum:all_spine_Hp} and~\ref{enum:single_leg_Hp} we call this edge $r(e)$. In case~\ref{enum:distinct_legs_Hp}, by using the leg $\ell$ of $C_{\mathsf{R'}}$ it intersects, we extend this edge to a path with an endpoint on the spine of $C_{\mathsf{R'}}$ and all its internal vertices on $\ell$. We also call this new path $r(e)$. After this, $r(e)$ has one endpoint on the chosen $(a+1)$-th horizontal path.

Given two subgraphs $F$ and $F'$ of $G$, we write $F \cup F'$ for the graph with vertices $V(F)\cup V(F')$ and edges $E(F)\cup E(F')$. 

In case~\ref{enum:distinct_legs_H}, the other endpoint of $r(e)$ is on a vertical tree $T(e)$ of $\mathsf{R}$. We extend $T(e)$ to a larger vertical tree $T(e) \cup  r(e)$. For distinct edges $e_1,e_2 \in M_3$, the vertical trees $T(e_1)$ and $T(e_2)$ are distinct and thus the extended vertical trees $T(e_1) \cup r(e_1)$ and $T(e_2)\cup r(e_2)$ are still vertex-disjoint. In this way we obtain $|M_3| \geq (2^m \cdot c)^{2^{m}}$ extended vertex-disjoint vertical trees that each intersect our chosen extra horizontal path. Thus we have constructed an $(a+1)\times |M_3|$-orchard, which contains an $(a+1)\times \left( (2^m \cdot c)^{2^{m}} \right)$-suborchard.

In case~\ref{enum:all_spine_H}, we do almost the same. The difference is that the $C_{\mathsf{R}}$-endpoint $v(e)$ of $r(e)$ is possibly not on a vertical tree.  In that case, in order to appropriately extend the vertical trees, we need to add some subpaths of the spine $P^*$ of $C_{\mathsf{R}}$. In doing that, we need to take care that the extended vertical trees are still vertex-disjoint. One can do this by ordering the vertices of $P^*$ `from left to right'. If $v(e)$ intersects a vertical tree, then we extend the tree as before. If $v(e)$ does not intersect a vertical tree, then we consider a tree $T(e)$ that has the closest intersection point with $P^*$  to the left of $v(e)$. There may exist (at most) one $e \in M_3$ such that $v(e)$ has no vertical tree strictly to its left. In that case we drop $e$ from $M_3$.
Next, we extend $T(e)$ to $T(e)\cup r(e)\cup p(e)$, where $p(e)$ is the smallest subpath of $P^*$ containing both $v(e)$ and a vertex of $T(e) \cap V(C_{\mathsf{R}})$.
Since each $r(e)$ meets a unique section of the horizontal path $P^*$ and since for every vertical tree $T$ there exist at most two horizontal sections of $P^*$ that intersect $T$ or are bordered by $T$ on their left, this ordering guarantees that at least half of the extended vertical trees remain pairwise vertex-disjoint. We thus obtain an $(a+1)\times \left \lfloor \tfrac{|M_3|-1}{2} \right \rfloor $-orchard, which contains a suborchard of the desired size since $|M_3| \geq  2\cdot (2^m \cdot c)^{2^{m}} +1$.

Note that the $(a+1)\times \left( (2^m \cdot c)^{2^{m}} \right)$-orchard that we have constructed in both cases is contained in $G[V(\mathsf{R}) \cup V(\mathsf{R'})]$. By \autoref{lem:tameorchard}, it contains a \textit{tame} $(a+1) \times \left(2^m \cdot c \right)$-suborchard, which  straightforwardly can be split into $2^m$ pairwise vertex-disjoint $(a+1) \times c$-orchards.
\end{proof}

We say that a subset $A$ of vertices of a graph $G$ \textit{reaches} a section $\mathsf{s}$ of an orchard $\mathsf{R}$ in $G$ if $G$ has a path from $A$ to $\mathsf{s}$ having no internal vertex in the orchard $\mathsf{R}$.

\begin{lemma}\label{lem:separateorchard_extended}
Let $m\in \mathbb{N}$. Suppose that $\mathsf{R}$ is an $a \times b$-orchard and $\mathsf{R'}$ is an $a' \times b'$-orchard vertex-disjoint from $\mathsf{R}$ in a graph $G$, with $a, a' \in [m]$.
Then, for each $c \geq 1$, at least one of the following holds.

\begin{itemize}
\item  $G$ contains $2^{m}$ pairwise vertex-disjoint $(a+1) \times c$-orchards;
\item there exists $X \subseteq V(G) \setminus V(\mathsf{R})$ with $|X| \leq f_{\sref{lem:separateorchard}}(c,m) + g_{\sref{lem:separateorchard}}(c,m)$ such that $V(\mathsf{R'}) \setminus X$ reaches at most $f_{\sref{lem:separateorchard}}(c,m) + g_{\sref{lem:separateorchard}}(c,m)$ sections of the orchard $\mathsf{R}$ in $G-X$.
\end{itemize}
\end{lemma}
\begin{proof}
Let $\mathcal{S}$ denote the set of sections of $\mathsf{R}$.
Recall that they are by definition vertex-disjoint.
Let $G'$ be the minor of $G$ obtained by contracting each path $\mathsf{s} \in \mathcal{S}$ into a single vertex, which we denote by $\bar{\mathsf{s}}$. Correspondingly, we write $\bar{\mathcal{S}}:= \left\{\bar{\mathsf{s}} \mid \mathsf{s} \in \mathcal{S}  \right\}$ for the set of contracted vertices in $V(G')$.

{\bf Case~1: There are $f_{\sref{lem:separateorchard}}(c,m) + g_{\sref{lem:separateorchard}}(c,m) + 1$ vertex-disjoint paths between $\bar{\mathcal{S}}$ and $V(\mathsf{R'})$ in $G'$.}
Then $G$ has a collection $\mathcal{M}$ of $f_{\sref{lem:separateorchard}}(c,m) + g_{\sref{lem:separateorchard}}(c,m) + 1$ vertex-disjoint paths, each having one endpoint in $V(\mathsf{R})$ and the other endpoint in $V(\mathsf{R'})$ and having no internal vertices in these two sets, such that the endpoints in $V(\mathsf{R})$ all belong to distinct sections of $\mathsf{R}$.

Let $G^*$ be obtained from the subgraph $\mathsf{R} \cup \mathsf{R'} \cup \bigcup_{P \in \mathcal{M}}P$ of $G$ by contracting each path $P\in \mathcal{M}$ into an edge joining its two endpoints.
Let $M^*$ denote the set of edges resulting from the contractions of the paths.
Thus $M^*$ is a matching.
Now, apply \autoref{lem:separateorchard} on $G^*$ with orchards $\mathsf{R}$ and $\mathsf{R'}$.
Since $|M^*| \geq f_{\sref{lem:separateorchard}}(c,m) + g_{\sref{lem:separateorchard}}(c,m) + 1$, the matching $M^*$ shows that the second outcome of that lemma is not possible.
Hence, we deduce that $G^*$ contains $2^{m}$ pairwise vertex-disjoint $(a+1) \times c$-orchards.
Replacing each edge of $M^*$ used in these orchards by the corresponding path in $\mathcal{M}$, we see that $G$ also has $2^{m}$ pairwise vertex-disjoint $(a+1) \times c$-orchards.

{\bf Case~2: There are at most $f_{\sref{lem:separateorchard}}(c,m) + g_{\sref{lem:separateorchard}}(c,m)$ vertex-disjoint paths between $\bar{\mathcal{S}}$ and $V(\mathsf{R'})$ in $G'$.}
By Menger's theorem, there is a subset $Z \subseteq V(G')$ of vertices with $|Z| \leq f_{\sref{lem:separateorchard}}(c,m) + g_{\sref{lem:separateorchard}}(c,m)$ separating $\bar{\mathcal{S}}$ from $V(\mathsf{R'})$ in $G'$. 
Let $\bar{\mathcal{S}'} := Z\cap \bar{\mathcal{S}}$ and $X:= Z \setminus \bar{\mathcal{S'}}$. 
Furthermore, let $\mathcal{S}'$ be the set of sections of $\mathcal{S}$ corresponding to vertices of $Z$. That is, let $\mathcal{S}':= \left\{ \mathsf{s} \mid \bar{\mathsf{s}} \in \bar{\mathcal{S'}} \right\}$.
Then, in the graph $G-X$, every path from $V(\mathsf{R'}) \setminus X$ to $V(\mathsf{R})$ enters $V(\mathsf{R})$ in a vertex belonging to some section $\mathsf{s}\in \mathcal{S'}$.
Thus, $V(\mathsf{R'}) \setminus X$ reaches at most $| \mathcal{S'}|$ sections of the orchard $\mathsf{R}$ in $G-X$.
Since $|X| \leq |Z|$ and $|\mathcal{S'}| = |\bar{\mathcal{S'}}|\leq |Z|$, and $|Z| \leq f_{\sref{lem:separateorchard}}(c,m) + g_{\sref{lem:separateorchard}}(c,m)$, the set $X$ has the desired properties.
\end{proof}

\begin{lemma}\label{lem:separatepathmix}
Let $m \in \mathbb{N}$. Suppose that $\mathsf{R}$ is an $a \times b$-orchard, with $a\in[m]$, in a graph $G$ and that $\mathsf{s}$ is a section of $\mathsf{R}$.
Then, for each $c \geq 1$, at least one of the following holds.
\begin{itemize}
\item$G$ contains $2^{m}$ pairwise vertex-disjoint $(a+1)
  \times c$-orchards;
\item there exists $X \subseteq V(\mathsf{s}) \cup (V(G) \setminus V(\mathsf{R}))$ with $|X| \leq 5f_{\sref{lem:separateorchard}}(c,m) + 5g_{\sref{lem:separateorchard}}(c,m)$ such that $V(\mathsf{s}) \setminus X$ reaches at most $5f_{\sref{lem:separateorchard}}(c,m) + 5g_{\sref{lem:separateorchard}}(c,m)$ sections of $\mathsf{R}$ distinct from $\mathsf{s}$ in $G-X$.
\end{itemize}
\end{lemma}

\begin{proof}
First, note that if $b=1$, then $\mathsf{R}$ has at most $3a$ horizontal sections and at most $a^2$ vertical sections, and the second outcome holds trivially with $X=\emptyset$ since $3a+a^2 \leq 3m + m^2 \leq 5f_{\sref{lem:separateorchard}}(c,m) + 5g_{\sref{lem:separateorchard}}(c,m)$.
Thus we may assume $b\geq 2$ in what follows.

  Suppose first that $\mathsf{s}$ is a section of some vertical tree $T$ of $\mathsf{R}$.
  Then we discard $T$ from $\mathsf{R}$ to obtain an $a\times(b-1)$-suborchard $\mathsf{R}_1$.
  Since $\mathsf{s}$ is disjoint from every horizontal path, $\mathsf{R}_1$ is vertex-disjoint from $\mathsf{s}$, while having the same horizontal paths as $\mathsf{R}$. Noting that $\mathsf{s}$ can be seen as a $1\times |\mathsf{s}|$-orchard $\mathsf{R'}$, we can apply \autoref{lem:separateorchard_extended} to $\mathsf{R}_1$ and $\mathsf{R'}$.
  Either we obtain $2^m$ pairwise vertex-disjoint
$(a+1)\times c$-orchards in $G$ (in which case we are done), or there
is a subset $X \subseteq V(G) \setminus V(\sR_1)$ such that $V(\sR') \setminus X$
reaches at most $f_{\sref{lem:separateorchard}}(c,m) +
g_{\sref{lem:separateorchard}}(c,m)$ sections of the orchard $\sR_1$ in
$G-X$, and $|X| \leq f_{\sref{lem:separateorchard}}(c,m) +
g_{\sref{lem:separateorchard}}(c,m)$. 
Note that on each horizontal path, there are at most three horizontal sections of $\sR$ that are not a section of $\sR_1$, namely: the unique section that has a non-empty intersection with $T$ and at most two sections that are bordered by $T$.
Since $T$ contains at most $a^2$ vertical sections, it follows that $V(\sR') \setminus X$ (which is $V(\mathsf{s}) \setminus X$)
reaches at most $f_{\sref{lem:separateorchard}}(c,m) +
g_{\sref{lem:separateorchard}}(c,m) + a^2 +3a \leq
5f_{\sref{lem:separateorchard}}(c,m) + 5g_{\sref{lem:separateorchard}}(c,m)$
sections of $\sR$ in $G-X$. Therefore $X$ has the desired property.

Next, suppose that $\mathsf{s}$ is a section of some horizontal path $P$ of $\mathsf{R}$. Decompose $P=P_0\mathsf{s}P_1$, where $P_0$ (respectively $P_1$) is the graph induced by the vertices of $P$ strictly to the left (respectively right) of $\mathsf{s}$.
For each $k \in \{0,1\}$, let $\mathsf{R}_{k}$ be the orchard obtained from $\mathsf{R}$ by discarding all vertical trees that intersect $P_{1-k}$ or $\mathsf{s}$, and truncating the horizontal path $P$ to $P-(V(P_{1-k}) \cup V(\mathsf{s}))$. Note that possibly $\mathsf{R}_k$ contains no vertical tree, in which case it has at most $a\leq m$ sections.

As before, we note that $\mathsf{s}$ forms a $1\times |\mathsf{s}|$-orchard $\mathsf{R'}$ that is
vertex-disjoint from $\mathsf{R}_k$.
We apply \autoref{lem:separateorchard_extended} to
$\mathsf{R}_k$ and $\mathsf{R'}$, for each $k\in \{0,1\}$. If $G$ does not contain $2^m$
pairwise vertex-disjoint $(a+1)\times c$-orchards, then for each $k\in \left\{0,1\right\}$, we obtain a subset $X_k \subseteq V(G) \setminus V(\mathsf{R}_k)$ with $|X_k| \leq
f_{\sref{lem:separateorchard}}(c,m) + g_{\sref{lem:separateorchard}}(c,m)$
such that $V(\mathsf{R'}) \setminus X_k$ reaches at most
$f_{\sref{lem:separateorchard}}(c,m) + g_{\sref{lem:separateorchard}}(c,m)$
sections of $\mathsf{R}_k$ in $G-X_k$. Observe that if $\mathsf{R}_k$ has no vertical tree, then we do not need to apply \autoref{lem:separateorchard_extended} since we can just take $X_k=\emptyset$. We now choose $X := X_0 \cup X_1$.

Possibly $\mathsf{s}$ is the intersection of a vertical tree $T$ and a horizontal path $P$. In that case we denote by $\mathsf{R}_T$ the orchard formed by $T$ and the horizontal sections of $\mathsf{R}$ that intersect $T$ or are bordered by $T$. Each vertical section of $\mathsf{R}$ is a vertical section of $\mathsf{R}_T$, $\mathsf{R_0}$ or $\mathsf{R}_1$. Note that $\mathsf{R}_T$ contains at most $a^2\leq m^2$ vertical sections and at most $3a\leq 3m$ horizontal sections.

Suppose a horizontal section $\mathsf{t}$ of $\mathsf{R}$ is not a horizontal section of $\mathsf{R}_0$, $\mathsf{R}_1$ or $\mathsf{R}_T$ (if defined). Then $\mathsf{t}$ must be bordered in $\mathsf{R}$ by a vertical tree of $\mathsf{R}_0$ and a vertical
tree of $\mathsf{R}_1$, and therefore we call $\mathsf{t}$ of \textit{mixed type}. Suppose $V(\mathsf{s}) \setminus X$ reaches $\mathsf{t}$ in $G-X$. Then it also reaches some horizontal section $\mathsf{t}^*$ of $\mathsf{R}_0$ or $\mathsf{R}_1$ in $G-X$ such that $\mathsf{t}$ is contained in  $\mathsf{t}^*$. Note that every horizontal section of $\mathsf{R}_0$ or $\mathsf{R}_1$ contains at most two horizontal sections of $\mathsf{R}$ that are of mixed type.

It follows from the previous discussion that $V(\mathsf{s}) \setminus X$ reaches at most $2\cdot 2 \cdot \left(f_{\sref{lem:separateorchard}}(c,m) +
g_{\sref{lem:separateorchard}}(c,m) \right) + m^2 + 3m \leq 5 f_{\sref{lem:separateorchard}}(c,m) +
5 g_{\sref{lem:separateorchard}}(c,m)$ sections of $\sR$ in $G-X$. As $|X|
\leq 5f_{\sref{lem:separateorchard}}(c,m) + 5g_{\sref{lem:separateorchard}}(c,m)$, we are done.
\end{proof}

Using Lemmas~\ref{lem:separateorchard_extended} and~\ref{lem:separatepathmix}, we derive the following lemma.

\begin{lemma}\label{lem:separ}
Let $m\in \mathbb{N}$. Suppose that $\mathsf{R}$ is an $a \times b$-orchard in a graph $G$ and $\mathsf{R'}$ is an $a' \times b'$-orchard in $G$ vertex-disjoint from $\mathsf{R}$, with $a, a'\in [m]$.
Then for each $c\geq 1$ at least one of the following holds.
\begin{enumerate}[label = (\arabic*)]
	\item $G$ contains a bramble of order at least $m$; \label{enum:grid}
	\item $G$ contains $2^{m}$ pairwise vertex-disjoint $(a+1) \times c$-orchards;  \label{enum:many}
	\item there exists $X \subseteq V(G)$ with
          \[
            |X| \leq f_{\sref{lem:separ}}(c,m):=(5f_{\sref{lem:separateorchard}}(c,m) + 5g_{\sref{lem:separateorchard}}(c,m)) \cdot (g_{\sref{lem:separ}}(c,m) + 1) + m \cdot \binom{g_{\sref{lem:separ}}(c,m)}{m}
          \]
          such that each component of $G-X$ that intersects $V(\mathsf{R'})$ intersects at most $g_{\sref{lem:separ}}(c,m):=(5f_{\sref{lem:separateorchard}}(c,m) + 5g_{\sref{lem:separateorchard}}(c,m))^m$ sections of the orchard $\mathsf{R}$. \label{enum:cutset}
\end{enumerate}
\end{lemma}
\begin{proof}
Let $\mathcal{S}$ denote the set of sections of $\mathsf{R}$.
Assume that~\ref{enum:many} does not hold (otherwise, we are done).
Then, for every section $\mathsf{s}\in \mathcal{S}$, \autoref{lem:separatepathmix} yields a subset $Y_{\mathsf{s}} \subseteq V(\mathsf{s}) \cup (V(G) \setminus V(\mathsf{R}))$ of size $|Y_{\mathsf{s}}|\leq 5f_{\sref{lem:separateorchard}}(c,m) + 5g_{\sref{lem:separateorchard}}(c,m)$ such that $\mathsf{s}$ reaches at most $5f_{\sref{lem:separateorchard}}(c,m) + 5g_{\sref{lem:separateorchard}}(c,m)$ sections of $\mathsf{R}$ in $G-Y_{\mathsf{s}}$.
Also, \autoref{lem:separateorchard_extended} gives a set $Y_r \subseteq V(G) \setminus V(\mathsf{R})$ of size at most $f_{\sref{lem:separateorchard}}(c,m)+g_{\sref{lem:separateorchard}}(c,m)$ such that $V(\mathsf{R'})$ reaches at most $f_{\sref{lem:separateorchard}}(c,m)+g_{\sref{lem:separateorchard}}(c,m)$ sections of $\mathsf{R}$ in $G-Y_r$.
Let $\mathcal{S}_r$ denote the set of sections of $\mathsf{R}$ reached by $V(\mathsf{R'})$ in  $G-Y_r$.

We construct an auxiliary directed graph $G^*$ with vertex set $\mathcal{S} \cup \{r\}$, where $r$ is a dummy element representing $\mathsf{R'}$, and adjacencies are defined as follows. 
For each $\mathsf{s} \in \mathcal{S}_r$, there is a directed edge from the vertex $r$ to $\mathsf{s}$.
For two distinct sections $\mathsf{s}, \mathsf{s}' \in \mathcal{S}$, there is a directed edge from $\mathsf{s}$ to $\mathsf{s}'$ if and only if $\mathsf{s}$ reaches $\mathsf{s}'$ in $G-Y_{\mathsf{s}}$.
It follows that the maximum outdegree of a vertex of $G^*$ is at most 
\[
  \max\{f_{\sref{lem:separateorchard}}(c,m) + g_{\sref{lem:separateorchard}}(c,m), 5f_{\sref{lem:separateorchard}}(c,m) + 5g_{\sref{lem:separateorchard}}(c,m)\} = 5f_{\sref{lem:separateorchard}}(c,m) + 5g_{\sref{lem:separateorchard}}(c,m).
\]
In what follows, vertices of $G^*$ will be classified by their {\em depth}, defined as the minimum length (number of directed arcs) in a directed path from $r$ to the vertex (or $+\infty$ in case no such directed path exists). 
Let $T^*$ be an out-arborescence obtained by performing a breadth-first search tree in $G^*$ from vertex $r$ using outgoing directed edges: 
For each section $\mathsf{s} \in \mathcal{S}$ at finite depth $d$, choose an in-neighbor of $\mathsf{s}$ with depth $d-1$ and add the corresponding directed edge to $T^*$. 
Note that $T^*$ only contains vertices of $G^*$ reachable from $r$ by a directed path, which might not be all vertices of $G^*$.
Define the {\em height} of $T^*$ as the maximum depth of a vertex of $T^*$. 
Let $\mathcal{S}_{\leq m}$ denote the set of sections $\mathsf{s} \in \mathcal{S}$ with depth at most $m$.

As a warm-up, suppose that the height of $T^*$ is less than $m$. 
Let $X := Y_r \cup \bigcup_{\mathsf{s} \in \mathcal{S}_{\leq m}} Y_{\mathsf{s}}$. 
Now, consider a path $P$ in $G-X$ having one endpoint in $V(\mathsf{R'})$ but no other vertex in $V(\mathsf{R'})$, and the other endpoint in a section $\mathsf{s} \in \mathcal{S}$. 
We claim that $\mathsf{s} \in \mathcal{S}_{\leq m}$. 
To see this, let us map the vertices of $P$ to vertices of $G^*$ in the expected way: Replace the endpoint of $P$ in $V(\mathsf{R'})$ by $r$, replace each maximal sequence of consecutive vertices of $P$ belonging to a section $\mathsf{s'} \in \mathcal{S}$ by the vertex $\mathsf{s'}$ of $G^*$, and remove all vertices of $P$ not in $V(\mathsf{R})$. 
This results in a sequence $r, \mathsf{s_1}, \mathsf{s_2}, \dots, \mathsf{s_k}$ of vertices of $G^*$ with $\mathsf{s_i} \in \mathcal{S}$ for each $i\in [k]$, some of which possibly appear multiple times. 
Now, observe that $V(\mathsf{R'})$ reaches $\mathsf{s_1}$ in $G-Y_{\mathsf{r}}$, so $(r, \mathsf{s_1})$ is a directed edge of $G^*$. 
Similarly, $\mathsf{s_i}$ reaches $\mathsf{s_{i+1}}$ in $G-Y_{\mathsf{s_i}}$ for each $i\in [k-1]$, so $G^*$ contains the directed edge $(\mathsf{s_i}, \mathsf{s_{i+1}})$. 
Hence, $r, \mathsf{s_1}, \mathsf{s_2}, \dots, \mathsf{s_k}$ is a directed walk in $G^*$, and therefore $\mathsf{s_k}=\mathsf{s}\in \mathcal{S}_{\leq m}$, since all vertices of $G^*$ with finite depth have depth less than $m$ by our assumption. 

It follows from the previous discussion that each component of $G-X$ intersecting $V(\mathsf{R'})$ intersects at most $|\mathcal{S}_{\leq m}| \leq g_{\sref{lem:separ}}(c,m)$ sections of $\mathsf{R}$, so that~\ref{enum:cutset} holds. 
Indeed, this number of sections is at most the number of vertices of
$T^*$,  which is bounded from above by
\[
  \Delta_{\text{out}}(T^*)^{\text{height}(T^*)+1} \leq (5f_{\sref{lem:separateorchard}}(c,m) + 5g_{\sref{lem:separateorchard}}(c,m))^{m} = g_{\sref{lem:separ}}(c,m),
\]
where $\Delta_{\text{out}}(T^*)$ denotes the maximum outdegree of $T^*$.

Moreover, we have 
\begin{align*}
  |X| &\leq f_{\sref{lem:separateorchard}}(c,m)
  +g_{\sref{lem:separateorchard}}(c,m) + (5f_{\sref{lem:separateorchard}}(c,m) + 5g_{\sref{lem:separateorchard}}(c,m)) \cdot g_{\sref{lem:separ}}(c,m)\\
& \leq (5f_{\sref{lem:separateorchard}}(c,m) + 5g_{\sref{lem:separateorchard}}(c,m)) \cdot (g_{\sref{lem:separ}}(c,m) + 1)\\
      &\leq f_{\sref{lem:separ}}(c,m).
\end{align*}
We may thus assume that the height of $T^*$ is at least $m$.

Let $Q\subset{V(T^*)}$ denote the set of sections with depth $m$.
Let $V_Q$ denote the set of vertices in $V(\mathsf{R})$ that are in a section in $Q$.
We now consider a maximum-size collection $\mathcal{Q}$ of vertex-disjoint paths that join $V(\mathsf{R'})$ with $V_Q$ in $G - \left (Y_r \cup \bigcup_{\mathsf{s}\in\mathcal{S}_{\leq m}} Y_{\mathsf{s}} \right )$ and we proceed with a case distinction on $|\mathcal{Q}|$, the number of these disjoint paths.

First, suppose $|\mathcal{Q}| \leq z(c,m):= f_{\sref{lem:separ}}(c,m)- (5f_{\sref{lem:separateorchard}}(c,m) + 5g_{\sref{lem:separateorchard}}(c,m)) \cdot (g_{\sref{lem:separ}}(c,m) + 1)$.
Then by Menger's Theorem, there is a set $C$ of vertices of size at most $z(c,m)$ separating $V(\mathsf{R'})$ from $V_Q$ in $G - \left (Y_r \cup \bigcup_{\mathsf{s}\in \mathcal{S}_{\leq m}} Y_{\mathsf{s}} \right )$. 
Observe that, by the definition of the directed graph $G^*$, every path in $G - \left (Y_r \cup \bigcup_{\mathsf{s}\in \mathcal{S}_{\leq m}} Y_{\mathsf{s}} \right )$ connecting a vertex of $V(\mathsf{R'})$ to a vertex belonging to a section of depth larger than $m$ must meet a section of depth exactly $m$. 
Thus, in the graph $G - \left (Y_r \cup \bigcup_{\mathsf{s}\in \mathcal{S}_{\leq m}} Y_{\mathsf{s}} \right )$, the set $C$ also separates $V(\mathsf{R'})$ from every vertex belonging to a section of depth larger than $m$. 
We then let $X:= C \cup Y_r \cup \bigcup_{\mathsf{s}\in \mathcal{S}_{\leq m}}
Y_{\mathsf{s}}$, which has size
\begin{align*}
  |X| &\leq z(c,m) + (5f_{\sref{lem:separateorchard}}(c,m) + 5g_{\sref{lem:separateorchard}}(c,m)) \cdot (g_{\sref{lem:separ}}(c,m) + 1)\\
  &\leq f_{\sref{lem:separ}}(c,m).
\end{align*}
As before, we find that each component of $G-X$ intersecting $V(\mathsf{R'})$ intersects at most $|\mathcal{S}_{\leq m}| \leq g_{\sref{lem:separ}}(c,m)$ sections of $\mathsf{R}$.

Next, assume that $|\mathcal{Q}|> z(c,m)$.
That is, there are many disjoint paths between $V(\mathsf{R'})$ and $V_Q$.
From this we will derive that $G$ contains a bramble of order at least $m$.
For each path $P \in \mathcal{Q}$, let the \textit{signature}
$\text{sign}(P)\subseteq \mathcal{S}$ of $P$ denote the set of the first $m$
different sections of $\mathcal{S}$ that $P$ intersects, starting from
its endpoint in $V(\mathsf{R'})$.

Note that $\text{sign}(P)\subseteq \mathcal{S}_{\leq m}$ and that it contains exactly $m$ elements, by construction.
Thus at most $|\mathcal{S}_{\leq m}| \leq g_{\sref{lem:separ}}(c,m)$ different sections can appear in signatures, and the number of distinct signatures is at most $\binom{g_{\sref{lem:separ}}(c,m)}{m}$.
By the pigeonhole principle it then follows that there is a set $\mathcal{P} \subseteq \mathcal{Q}$ of 
\[\frac{z(c,m)}{\binom{g_{\sref{lem:separ}}(c,m)}{m}}=m\] 
disjoint paths that have a common signature $\mathcal{T}$. 
By definition of signature, each $P \in \mathcal{P}$ and $T \in \mathcal{T}$ have at least one vertex in common. Therefore $\mathcal{B} :=\left\{ T \cup P \mid  (T,P) \in \mathcal{T} \times \mathcal{P} \right\}$ is a bramble. Moreover, $\mathcal{B}$ has order at least $m$ since the paths in $\mathcal{P}$ (respectively $\mathcal{T}$) are vertex-disjoint and $|\mathcal{P}| = |\mathcal{T}| = m$. This is outcome \ref{enum:grid} so we are done. 
\end{proof}

\section{Packing orchards}\label{sec:packing_orchards}

This section deals with packings of orchards of prescribed types in a graph.
Let $G$ be a graph, let $m\in \N$, and let $\omega \colon [m] \to \N$ be a decreasing function.
An \emph{orchard $(m, \omega)$-packing} in $G$ is a tuple $\mathcal{D} = (\mathcal{R}_1, \dots, \mathcal{R}_m)$ such that
\begin{itemize}
\item for every $i \in [m]$, $\mathcal{R}_i$ is a (possibly empty) collection of $i \times \omega(i)$-orchards;
\item all orchards in $\bigcup_{i=1}^m \mathcal{R}_i$ are pairwise vertex-disjoint;
\item every orchard $\sR \in \mR_1$ is a path with exactly $\omega(1)$ vertices.
\end{itemize}

We write $V(\mathcal{D})$ for the vertex set $\bigcup_{i=1}^m V(\mathcal{R}_i)$.
The \emph{grade} of $\mathcal{D}$ is the sum $\sum_{i=1}^m 2^i|\mathcal{R}_i|$.
We say that $\mathcal{D}$ is {\em optimal} if it has maximum grade among all orchard $(m, \omega)$-packings in~$G$.

\begin{lemma}\label{lem:prop_optimal_packing}
Let $m$ and $\omega$ be as above, let $\mathcal{D}=(\mR_1, \dots, \mR_m)$ be an optimal orchard $(m, \omega)$-packing in a graph $G$, and let $\mathsf{R} \in \mathcal{R}_i$ and $\mathsf{R'}\in \mathcal{R}_j$ for some $i,j\in [m]$.
Then $G':=G[V(\mathsf{R}) \cup V(\mathsf{R'}) \cup (V(G) \setminus V(\mathcal{D}))]$ does not contain $2^m$ pairwise vertex-disjoint $(i+1)\times \omega(i+1)$-orchards.
\end{lemma}
\begin{proof}
Suppose $G'$ \textit{does} contain $2^m$ pairwise vertex-disjoint
$(i+1)\times \omega(i+1)$-orchards $\mathsf{R}_1, \ldots,
\mathsf{R}_{2^m} $. Then we can obtain a new orchard $(m, \omega)$-packing
$\mathcal{D}'$ from $\mathcal{D}$ by removing $\mathsf{R}$ and $\mathsf{R}'$ from $\mathcal{R}_i$ and $\mathcal{R}_j$, respectively, and adding $\mathsf{R}_1,
\ldots, \mathsf{R}_{2^m}$ to $\mathcal{R}_{i+1}$.
Any other orchard of $\mathcal{D}$ is vertex-disjoint from $G'$ and
is therefore unaffected by this replacement. It follows that the grade
has been increased by $2^{i+1} \cdot 2^m \geq 2^{m+2}$ and has been
decreased by $2^{i} + 2^{j} \leq 2^{m+1}$. Thus $\mathcal{D}'$ has a
higher grade than $\mathcal{D}$, a contradiction.
\end{proof}	

\begin{lemma}\label{lem:optimal_packing_bounded_model}
Let $m$ and $\omega$ be as above, let $\mathcal{D}=(\mR_1, \dots, \mR_m)$ be an optimal orchard $(m, \omega)$-packing in a graph $G$, and let $Z \subseteq V(G)$.
Let $q$ be the number of orchards in $\bigcup_{i=1}^{m} \mR_i$ having at least one vertex in common with $Z$.
Then, for every minor $H$ of $G[Z]$, there is a model of $H$ in $G[Z]$ of size at most
$$\|H\| \cdot \max(q,1) \cdot 2^{m+1} \cdot  \omega(1).$$
\end{lemma}
\begin{proof}
First we control the length of paths in $G[Z]$. If $q=0$, then $G[Z]$ does not contain a $1\times \omega(1)$-orchard. Otherwise, this orchard would not intersect any other orchard of $\mathcal{D}$, so we could add it to $\mR_1$, which would contradict the maximality of the grade of $\mathcal{D}$.
So each path in $G[Z]$ has order smaller than $\omega(1)$ when $q=0$.
Suppose, on the other hand, that $q\geq 1$ and $G[Z]$ contains a path $P$ of order $q2^{m+1} \cdot \omega(1)$.
Then $P$ can be split into $q2^{m+1}$ vertex-disjoint paths of order $\omega(1)$, each of which can be viewed as a $1\times \omega(1)$-orchard.
We add these $q2^{m+1}$ orchards to $\mR_1$ after first deleting from $\bigcup_{i=1}^{m} \mR_i$ the $q$ orchards intersected by $Z$, thus obtaining a new orchard $(m, \omega)$-packing.
This replacement increases the grade by $q2^{m+1}$ and decreases it by
at most $q2^{m}$. Hence the new packing has higher grade than $\mathcal{D}$, which contradicts the maximality of the grade of $\mathcal{D}$.
Thus each path in $G[Z]$ has order smaller than $q2^{m+1} \cdot \omega(1)$.

Let now $H$ be a minor of $G[Z]$ and let $\mathcal{M}:=\left\{M_x \subseteq G[Z] : x \in V(H) \right\}$ be a corresponding model of $H$ in $G[Z]$.
 For each $x \in V(H)$ we choose a vertex $v(x)$ in $V(M_x)$. For every edge $xy$ in $E(H)$, we choose a shortest path $P_{xy}$ in $G[V(M_x)\cup V(M_y)]$ with endpoints $v(x)$ and $v(y)$. We obtain a new model $\mathcal{M}'=\left\{M_x \cap \left( \bigcup_{xy \in E(H)} V(P_{xy}) \right)  : x \in V(H) \right\}$ of $H$ in $G[Z]$.
 As $V(\mathcal{M}')$ can be covered by $\|H\|$ paths of $G[Z]$, it
 follows that it has size $|V(\mathcal{M}')| < \|H\| \cdot \max(q,1) \cdot 2^{m+1} \cdot \omega(1)$.
\end{proof}	

\section{Proof of \autoref{thm:main_technical}}
\label{sec:tech}

Now that optimal orchard packings are defined, we will use a strategy adapted from the proof for wheel minors in~\cite{ErdosPosaWheels} to show our main technical theorem, \autoref{thm:main_technical}.
(For readers familiar with~\cite{ErdosPosaWheels}, our orchards will play the roles of the bounded-size paths and cycles in that proof.) 

\newcommand{\sizeofer}{\|H\|}

We start with a brief overview of the proof. First, we will define several constants and functions, among which are the constant $m$ and the function $\omega$, that only depend on the given parameters $p, H$ and $g$. Next, we choose an arbitrary graph $G$ and we consider an optimal $(m,\omega)$-orchard packing $\mathcal{D}=(\mathcal{R}_1, \dots, \mathcal{R}_m)$ of $G$. We also need to take into account the components of $G-V(\mathcal{D})$, but for this proof sketch we will assume that there are no such components. We construct two auxiliary graphs $\Gbig$ and $\Gsmall$ to derive either a small $K_p$-model (in which case we are done, having obtained outcome~\ref{e:shallow-clique}) or: an orchard $K$ in $\bigcup_{i} \mathcal{R}_i$ that only sees a small number of other orchards of $\bigcup_{i} \mathcal{R}_i$. Next, for each orchard $K'$ that is seen by $K$, we consider the graph $G_{K'}$ induced by $V(K)\cup V(K') \cup (V(G) \setminus V(\mathcal{D}))$. Using~\autoref{lem:separ} and the optimality properties of $\mathcal{D}$, we find a small set $X_{K'}$ of vertices such that each component of $G_{K'}-X_{K'}$ intersecting $V(K')$ only intersects a small number of sections of $K$ (otherwise we obtain a small model of $H$, satisfying outcome~\ref{e:smallmodel}). We define the cutset $X:= \bigcup_{K'}X_{K'}$ and we finish the argument by deriving a suitable separation $(A,B)$ with $X=A\cap B$, satisfying outcome~\ref{e:sep}. This concludes the proof sketch.

\begin{proof}[Proof of \autoref{thm:main_technical}] 
  We use the following functions or constants from previously stated lemmas and theorems:\hypertarget{anchor:def-varphi}{}
  \begin{itemize}
  \item $\varphip,\varphipa\in \R$ are
    constants depending only on $p$ such that every
    $n$-vertex graph of average degree at least $\varphip$ has a
    $K_{p}$-model on at most $\varphipa \log n$ vertices
    (see~\autoref{thm:small-minors});
  \item $m := f_{\sref{th:gridminor}}(|H|)+1$.
  \end{itemize}

 \hypertarget{anchor:def-omega}{}
 We define a decreasing function $\omega \colon [m] \to \N$ as follows.
 We set $\omega(m) := m$ and, for every
$i = m-1, \dots, 1$,
\begin{align*}
  q(i) &:= 2\varphip^2 f_{\sref{lem:separ}}(\omega(i+1),m)\\[.5ex]
  &\text{and}\\[.5ex]
  \omega(i)&:=  (q(i) + 1 )
             \cdot  \max \left \{\ g(q(i)) + 1,\
    g_{\sref{lem:separ}}(\omega(i+1),m) + 1\ \right \}.
\end{align*}

By straightforward calculation and the facts that $\varphip \geq 1$ and $g(q(m-1))\geq g(0)=1$, it follows that $\omega(m-1) > m =\omega(m)$.  The functions $f_{\sref{lem:separ}}(.,m)$ and $g_{\sref{lem:separ}}(.,m)$ are increasing in their first coordinate, while $g(.)$ is non-decreasing.  Therefore  $\omega(i+1)> \omega(i+2)$ implies $\omega(i)>\omega(i+1)$, for every $i=m-2,\ldots,1$. We conclude that $\omega(i)$ is indeed a decreasing function.
  
  Let $\alpha  :=  2^{m+1} \omega(1)$.
  We prove the theorem with
  \[
    \sigma: =\alpha  \cdot \max \left \{\ p^2\varphipa ,\
                              \sizeofer{} \cdot (2\varphip^2
                                f_{\sref{lem:separ}}(\omega(1),m) + 1)\ \right\}.
  \]

We remark that $\varphip$ and $\varphipa$ depend only on $p$, while $m=\omega(m)$ depends only on $H$. Furthermore (for $i \neq m$) each of $q(i), \omega(i), \alpha$ and $\sigma$ depends on $p$, $H$ and $g$.

Let $G$ be a graph. Let us assume that $G$ does not contain an $H$-model of size at most $\sigma$.
We show that one of the two other outcomes of the theorem holds.
Let $\mathcal{D} = (\mathcal{R}_1, \dots, \mathcal{R}_m)$ be an optimal orchard $(m, \omega)$-packing in $G$ (this is well-defined because $\omega$ is a decreasing function).
We call a graph a $\emph{piece}$ if it is an orchard from one of the collections $\mathcal{R}_1, \dots, \mathcal{R}_m$, or if it is a component of the graph $G- V(\mathcal{D})$.

Suppose some piece $K$ contains a model of $H$. By an application of \autoref{lem:optimal_packing_bounded_model} with $Z=V(K)$ and $q \in \left\{0,1\right\}$, it follows that $K$ (and hence $G$) contains a model of $H$ of size at most $\sizeofer{} \cdot \alpha  \leq \sigma$, contradicting our initial assumption. Therefore every piece is $H$-minor free.
We recall that the pieces of $\mathcal{D}$ are orchards and thus subgraphs of $G$ that are not necessarily induced, whereas the other pieces are induced subgraphs.

Suppose $\mathcal{D}$ contains at most one piece. Then from \autoref{lem:optimal_packing_bounded_model} applied with $Z=V(G)$ and $q\in \left\{0,1\right\}$, we know
that if $G$ has a model of $H$, then it has one of
size at most $\sizeofer{}\cdot  \alpha  \leq \sigma$. This cannot happen, because of
our initial assumption.
If $G$ has no such model, then we can take $(A, B)$ with $A=V(G)$
and $B=\emptyset$ as a trivial separation, which satisfies the outcome
\ref{e:sep} of the theorem because $|A|=|V(G)| \geq 1 = g(0) = g(|A\cap B|)$.
Thus we may assume from now on that $\mathcal{D}$ contains at least two pieces.

A piece is said to be \emph{central} if it belongs to $\mathcal{D}$, or if it sees at least $2\varphip$ other pieces.
(Note that pieces not in $\mathcal{D}$ do not see each other, by definition.)
In the next paragraph, we define two auxiliary graphs $\Gsmall$ (for small degrees) and $\Gbig$ (for big degrees) that model how the central pieces are connected through the noncentral pieces. To keep track of the correspondence between the edges of $\Gsmall$ and the noncentral pieces, we put labels on some of these edges.

Initialize both $\Gsmall$ and $\Gbig$ to the graph whose set of vertices is the set of central pieces and whose set of edges is empty. For each pair of central pieces that see each other in $G$, add an (unlabeled) edge between the corresponding vertices in both $\Gsmall$ and $\Gbig$.

Next, while there is some noncentral piece $N$ that sees two central pieces that are not yet adjacent in $\Gbig$, do the following two operations:
\begin{enumerate}
\item Add all (unlabeled) edges to $\Gbig$ between pairs of central pieces seeing $N$ (not already present in $\Gbig$). This creates a clique on the set of central pieces seeing $N$ in $\Gbig$, some of whose edges might have already been there before.
\item Then, among the central pieces seeing $N$, choose one such piece $K$ such that the number of newly added edges of $\Gbig$ incident to $K$ is maximum. Add to $\Gsmall$ every edge that links $K$ to another central piece seeing $N$ (not already present in $\Gsmall$), and label it with $N$. This creates a star centered at $K$ in $\Gsmall$ with all its edges labeled with the noncentral piece $N$.
\end{enumerate}

By construction, $\Gsmall$ is a subgraph of $\Gbig$ (if we forget
about labels). These graphs have the following two crucial properties.

\begin{claim}\label{claim:model}
If $\Gsmall$ has a $K_p$-model of size $\ell$, then $G$ has a $K_p$-model
of size at most $\alpha \ell p^2$.
\end{claim}

\begin{proof}
  \newcommand{\Msmall}{M_{\mathrm{s}}}
  Suppose that $\Gsmall$ has a $K_p$-model of size $\ell $. Then there
  exists a subgraph $\Msmall \subseteq \Gsmall$ with $\ell $ vertices that
  can be contracted to $K_p$.
  Let $Z$ be the union of $V(K)$ over all central pieces $K \in
  V(\Msmall)$ and all pieces $K$ not in $\mathcal{D}$.
  It follows from the construction of $\Gsmall$ that $G[Z]$ contains a graph isomorphic to $\Msmall$ as a minor, and thus has a $K_p$ minor.
  As $Z$ intersects at most $\ell $ central pieces,
  \autoref{lem:optimal_packing_bounded_model} implies that $G[Z]$
  has a model of $K_p$ of order at most~$2^{m+1}\omega(1) \cdot \ell
  \cdot \|K_p\| \leq \alpha \ell p^2$.
  \cqed%
\end{proof}

For a graph $F$, we denote by $\overline{d}(F)$ the average degree of $F$.

\begin{claim}\label{claim:av}
The average degrees of $\Gbig$ and $\Gsmall$ satisfy the inequality $\overline{d}\left( \Gbig\right) \leq \varphip \cdot \overline{d}\left( \Gsmall\right)$. Moreover, the degree in $\Gsmall$ of each central piece not in $\mathcal{D}$ is at least $2\varphip$.
\end{claim}
\begin{proof}
First, note that edges that appear in $\Gbig$ but not in $\Gsmall$ must not be labeled.
Let $N$ be a noncentral piece, and let $r$ be the number of pieces in $\mathcal{D}$ it sees. By definition of noncentral pieces,  $r <2\varphip$. When $N$ is treated in the algorithm used to construct $\Gbig$ and $\Gsmall$, if $\ell $ new edges are added to $\Gbig$, then one of the pieces seen by $N$ is incident to at least $2\ell /r>\ell  / \varphip$ of these new edges and thus at least $\ell  / \varphip$ new edges are added to $\Gsmall$. This proves the first part of the claim.

By definition, a piece $K$ not in $\mathcal{D}$ is central if it sees at least $2\varphip$ other pieces. As pieces not in $\mathcal{D}$ do not see each other, $K$ sees at least $2\varphip$ pieces from $\mathcal{D}$, that is, at least $2 \varphip$ other central pieces. Then in the first step of the construction of $\Gsmall$, all edges have been added from $K$ to these pieces.
\cqed%
\end{proof}

If $\overline{d}\left(\Gsmall\right) \geq \varphip$, then by
\hyperlink{anchor:def-varphi}{definition} of $\varphip$ and $\varphipa$ at the beginning of the proof,
$\Gsmall$ has a $K_{p}$-model of size at most $\varphipa\log
|\Gsmall|$. By \autoref{claim:model}, this gives a $K_p$-model of size
at most $\varphipa\alpha p^2 \log |\Gsmall| \leq \sigma \log |G|$ in $G$
and we are done (outcome \ref{e:shallow-clique}).

Thus, we assume in the rest of the proof that $\overline{d}\left(\Gsmall\right)<\varphi$.
Then strictly more than half of the central pieces have degree less
than $2\varphip$ in $\Gsmall$ (otherwise at least half of the vertices
of $\Gsmall$ have degree at least $2\varphi$, a contradiction to the
fact that $\overline{d}\left(\Gsmall\right)< \varphip$). By
\autoref{claim:av}, we obtain $\overline{d}\left(\Gbig\right)<\varphip^2$ and we similarly get that more than half of the central pieces have degree less than $2\varphip^2$ in $\Gbig$.
Since $\mathcal{D}$ is nonempty, it follows that there is a central piece whose degree in $\Gsmall$ is less than $2\varphip$, and whose degree in $\Gbig$ is less than $2\varphip^2$. Choose such a piece $K$. By \autoref{claim:av} (second part of the statement), $K$ is in $\mathcal{R}_i$ for some $i\in [m]$.
That is, $K$ is an $i\times \omega(i)$-orchard.
As observed in the beginning of the proof, every piece is $H$-minor free. By
\autoref{lem:big-orchard} this implies $i<m$; in particular $\omega(i+1)$ is defined.

The rest of the proof relies on the fact that $K$ has degree less than $2\varphip^2$ in $\Gbig$. We will not use the graph $\Gsmall$ anymore.

Recall that $\mathcal{D}$ contains at least two pieces.
For each piece $K'$ in $\mathcal{D}$ adjacent to $K$ in $\Gbig$, let $G_{K'}$ be the subgraph of $G$ induced by $V(K) \cup V(K') \cup (V(G) \setminus V(\mathcal{D}))$.
Apply \autoref{lem:separ} with orchards $K$ and $K'$ on the graph
$G_{K'}$  with $c=\omega(i+1)$.
According to \autoref{lem:prop_optimal_packing}, the outcome
\ref{enum:many} of \autoref{lem:separ} is not possible.
If outcome \ref{enum:grid} holds, that is if $G_{K'}$
contains a bramble of order at least $m$, then by
\autoref{th:brambleduality} and \autoref{th:gridminor} $G_{K'}$
contains a model of $H$. By
\autoref{lem:optimal_packing_bounded_model} (applied with $Z=V(G_{K'})$ and $q=2$), there is such a model of
size at most $2\sizeofer{} \cdot \alpha  \leq \sigma$, a contradiction.
Therefore we may assume that we get outcome \ref{enum:cutset} when applying \autoref{lem:separ}. So, there is a set $X_{K'}$ of vertices of size at most $f_{\sref{lem:separ}}(\omega(i+1),m)$ such that each component of $G_{K'} - X_{K'}$ that intersects $V(K')$ intersects at most $g_{\sref{lem:separ}}(\omega(i+1),m)$ sections of the orchard $K$.

Let $X := \bigcup_{K'} X_{K'}$, where the union is taken over all pieces $K'$ in $\mathcal{D}$ adjacent to $K$ in $\Gbig$.
Then $|X| \leq 2\varphip^2 f_{\sref{lem:separ}}(\omega(i+1),m) =: q$. Note
that $q$ coincides with $q(i)$ defined at the beginning of the proof.

Also, from the definition of $q$ and $\omega$ we have
\[
  \frac{\omega(i)-q}{q+1} > g_{\sref{lem:separ}}(\omega(i+1),m)\quad
  \text{and}\quad \frac{\omega(i)-q}{q+1} > g(q).
\]

Now, consider some horizontal path of $K$.
Recall that there are at least $\omega(i)$ horizontal sections on that path, since every vertical tree defines one such section.
By the pigeonhole principle, we can find $(\omega(i)-q)/(q + 1)$ consecutive horizontal sections that are avoided by $X$.
Let $P$ denote the subpath of the horizontal path induced by the vertices of these sections.

Let $C$ be the component of $G-X$ that contains $P$.
We claim that no orchard $K'$ in $\mathcal{D}$ distinct from $K$ has a vertex in $C$.
Suppose for a contradiction that an orchard $K'$ does, and let $Q$ be a path in $C$ having one endpoint in $P$ and the other endpoint in $K'$.
By choosing $K'$ appropriately, we can moreover ensure that $Q$ does not intersect any other orchard distinct from $K$ and $K'$.
It follows that $P \cup Q$ is a subgraph of $G_{K'} - X_{K'}$.
Since $P \cup Q$ is connected, it is contained in some component of $G_{K'} - X_{K'}$.
Hence, that component intersects $K'$ and at least
$
  (\omega(i)-q)/(q+1) > g_{\sref{lem:separ}}(\omega(i+1),m)
$
sections of $K$, contradicting \autoref{lem:separ}.

Let $A := X \cup V(C)$ and $B := V(G) \setminus V(C)$.
Since $C$ intersects no orchard from $\mathcal{D}$ other than $K$, it follows that $A$ intersects at most $|X|+1 \leq q+1$ orchards from $\mathcal{D}$.
If $G[A]$ has $H$ as minor, then $G[A]$ has a model of size at
most $\sizeofer{} (q+1)2^{m+1} \cdot \omega(1)\leq \sigma$ by
\autoref{lem:optimal_packing_bounded_model}, a contradiction.
If $G[A]$ has no $H$-model, then $(A, B)$ is a separation of $G$
with the desired properties, since $|A| \geq |P| \geq
\frac{\omega(i)-q}{q+1} \geq g(q)$
and $|A\cap B| = |X| \leq q$.
\end{proof}

\section{Approximation algorithm}
\label{sec:apalg}

The statements and proofs of \autoref{thm:main} and \autoref{thm:main_technical} were described without mentioning algorithmic aspects.
In this section, we briefly explain how the different steps of the proofs can be made algorithmic, and thus obtain \autoref{cor:approx}.

First, let us address one subtlety, namely that the constant $c$ in our $c k \log (k+1)$ bound in \autoref{thm:main} is not known to be computable.
This is because $c$ depends on the polynomial $p_{\sref{thm:fominkernel}}$ corresponding to $H$ in \autoref{thm:fominkernel}, which is not known to be computable (see the remarks at the end of the kernelization section in~\cite{FominFdel}).
Nevertheless, this does not prevent us from deriving the approximation algorithm, as we will explain. 
(We also note that variants of \autoref{thm:fominkernel} have recently been developed in~\cite{JansenPieterse18} with computability of the constants as an explicit goal; however, these results need extra assumptions on the graph and are not applicable in our context.) 

First we explain how to obtain the algorithm in \autoref{cor:approx}, assuming we have an algorithm for \autoref{thm:main_technical}, and then we explain how the proof of \autoref{thm:main_technical} can be made algorithmic. That is, we assume that for every $p$, $H$, and $g$ as in the statement of \autoref{thm:main_technical}, there is a constant $\sigma \in \N$ and a polynomial-time algorithm that, given a graph, returns one of the three objects promised by~\autoref{thm:main_technical}.

Our algorithm will use the following algorithmic version of \autoref{thm:fominkernel} given in~\cite{FominFdel} (see the paragraph before Theorem~15 in this paper for details).
We remark that the polynomial $p_{\sref{thm:kernelalgo}}$ appearing in the statement is in fact the same as $p_{\sref{thm:fominkernel}}$ but we distinguished them to avoid confusion.

\begin{theorem}[Fomin, Lokshtanov, Misra, and Saurabh {\cite[Lemma~25]{FominFdel}}]
\label{thm:kernelalgo}
 For every fixed planar graph $H$, there is a polynomial $p_{\sref{thm:kernelalgo}}$ and a polynomial-time algorithm $\mathcal{A}$ which, given a pair $(G, t)$ as input, where $G$ is a graph and $t \geq 0$ is an integer,
\begin{itemize}
	\item either produces a minor $G'$ of $G$ such that $\tau_H(G') = \tau_H(G)$ and $|G'| \leq p_{\sref{thm:kernelalgo}}(t)$, together with the sequence of operations (edge/vertex deletions, edge contractions) used to obtain $G'$ from $G$,
	\item or (correctly) answers that $\tau_H(G) > t$.
\end{itemize}
Moreover, in the first case, given any set $X'\subseteq V(G')$ such that $G'-X'$ is $H$-minor-free, the algorithm can compute a corresponding set $X\subseteq V(G)$ of the same size such that $G-X$ is $H$-minor-free in polynomial time.
\end{theorem}

We call the operation of replacing $X'$ by $X$ as {\em lifting} $X'$ to $G$.
Note that, given any packing $\mathcal{P'}$ of $H$-models in $G'$, one can also easily compute a corresponding packing $\mathcal{P}$ of $H$-models in $G$ of the same size in polynomial time, because the algorithm above provides the sequence of operations used to obtain $G'$ from $G$.
We call this \emph{lifting} the packing $\mathcal{P'}$ to $G$.

We will also need an algorithmic version of \autoref{thm:fjw13}, which is provided in~\cite{FJW2013}: There is a polynomial-time algorithm $\mathcal{B}$ which, given the graph $G$ and the separation $(A, B)$ of bounded order, computes the graph $G'$ with $\nu_H(G') = \nu_H(G)$ and $\tau_H(G') = \tau_H(G)$ guaranteed by that theorem.
Furthermore, a given packing $\mathcal{P'}$ of $H$-models in $G'$ can be lifted to $G$ in polynomial time, and the same is true for a given subset $X'\subseteq V(G')$ such that $G'-X'$ is $H$-minor-free. We refer to the first paragraph of Section 5 of \cite{FJW2013} for more details about these algorithms.

Fix a planar graph $H$.
Let us assume for now that $H$ is connected, we will comment on the disconnected case later.
Let $g$ denote the function $f_{\sref{thm:fjw13}}$ of \autoref{thm:fjw13} for the graph $H$.
Our algorithm for \autoref{cor:approx} is a recursive algorithm $\mathcal{R}$ which, given the input graph $G$, outputs a packing $\mathcal{P}$ of $k$ $H$-models in $G$ and a subset $X$ of vertices of $G$ such that $G-X$ has no $H$-model, of size at most $c' k \log (k+1)$, where $c':=c'(H)$ is a constant depending on the constant $c$ in \autoref{thm:main} and $k$ is some positive integer.
The algorithm is as follows.

\begin{itemize}
	\item Find the smallest integer $t\in \{0, 1, \dots, |G|\}$ such that algorithm $\mathcal{A}$ on $(G, t)$ does not report $\tau_H(G) > t$.
	\item Then $t \leq \tau_H(G)$ because either $t=0$, or $t>0$ and $\mathcal{A}$ reported $\tau_H(G) > t-1$ on $(G, t-1)$.
	\item Let $G'$ be the minor of $G$ output by $\mathcal{A}$ on $(G, t)$, which satisfies $|G'| \leq p_{\sref{thm:kernelalgo}}(t)$.
	\item If $G'$ is empty, stop and output $(\emptyset, \emptyset)$ for the pair $(\mathcal{P},X)$.
	\item Run on $G'$ the algorithm corresponding to \autoref{thm:main_technical} with parameters $p:= |H|$, $H$, and $g$.
    \item {\bf If the output is an $H$-model or a $K_p$-model $\mathcal{M'}$:}
    \begin{itemize}
    	\item Let $G'' := G' - V(\mathcal{M'})$.
        \item Run algorithm $\mathcal{R}$ on $G''$, let $\mathcal{P''}$ and $X''$ denote the packing and subset of vertices it outputs.
        \item Let $\mathcal{P'} := \mathcal{P''} \cup \{ \mathcal{M'} \}$ and $X':= X''\cup V(\mathcal{M'})$.
	\end{itemize}
	\item {\bf Else, the output is a separation $(A,B)$}:
	\begin{itemize}
    	\item Apply algorithm $\mathcal{B}$ on $G'$ with separation $(A, B)$, producing a graph $G''$ with $\nu_H(G'')=\nu_H(G')$ and  $\tau_H(G'')=\tau_H(G')$.
    	\item Run algorithm $\mathcal{R}$ on $G''$, let $\mathcal{P''}$ and $X''$ denote the packing and subset of vertices it outputs.
        \item Lift $\mathcal{P''}$ to a packing $\mathcal{P'}$ in $G'$.
        \item Lift $X''$ to a set $X'$ of vertices of $G'$.
	\end{itemize}
      \item Lift $\mathcal{P'}$ to a packing $\mathcal{P}$ in $G$.
      \item Lift $X'$ to a set $X$ of vertices of $G$.
      \item Output $\mathcal{P}$ and $X$.
\end{itemize}

Note that above we did not distinguish between the first two outcomes of \autoref{thm:main_technical} because in each case we obtain an $H$-model of order $O(\log|G'|)$.
To avoid confusion, let us write $G_i, G_i'$ and $G_i''$ for respectively the graphs $G, G'$ and $G''$ after the $i$-th recursive call, thus $G_0$ is our original graph $G$.
Let also $\tau:=\tau_H(G)$.
Observe that for all $i$, we have $\tau_H(G_i) \leq \cdots \leq \tau_H(G_0)=\tau$, and thus  $|G_i'| \leq p_{\sref{thm:kernelalgo}}(\tau_H(G_i)) \leq  p_{\sref{thm:kernelalgo}}(\tau)$.

Hence, when an $H$-model $\mathcal{M'}$ of $G_i'$ is considered in the
algorithm, it satisfies $|V(\mathcal{M'})| \leq \sigma \log |G_i'|
\leq \sigma' \log(\tau + 1)$, for some constant $\sigma'$ depending on $\sigma$ and $p_{\sref{thm:kernelalgo}}$.
Letting $k:=|\mathcal{P}|$, it follows that $|X| \leq \sigma' k \log
(\tau + 1)$, where $\mathcal{P}$ and $X$ are the packing and subset of vertices output after the initial call of the algorithm.

Therefore, it only remains to show that $|X|\leq c' k
\log (k+1)$ for some constant $c'$. This clearly holds when $X$ is
empty, so we assume from now on that $|X| \geq 1$.
Note also that $|X| \geq \tau$.

We first observe that for every real $x > 0$ we have $2 x \geq 3 \log x$.
By substituting $x$ by $\log (x+1)$, multiplying by $\sigma' k$, and
rearranging the terms we deduce that for every $x > 0$ the following holds:
\begin{equation}
0 \leq 2\sigma'k \log (x + 1) - 3\sigma'k \log \log (x + 1). \label{eq:log-magic}
\end{equation}
Coming back to the size of $X$, we have:
\begin{align}
  |X| &\leq \sigma' k \log (\tau + 1) \nonumber \\
      &\leq \sigma' k \log (|X| + 1) \label{eq:xbound}\\
      &\leq 3\sigma' k \log (|X| + 1) - 3 \sigma' k \log \log (|X| + 1)\label{eq:minus}\\
      &\leq 3\sigma' k \log \left [ \sigma'k \log (|X| + 1) +1 \right ] - 3 \sigma' k \log \log (|X| + 1)\label{eq:reduce}\\
      &\leq 3\sigma' k \log \left [ \sigma'(k+1) \cdot \log (|X| + 1) \right ] - 3 \sigma' k \log \log (|X| + 1)\label{eq:expand}\\
      &\leq 3\sigma' k \log (\sigma'(k+1)). \nonumber
\end{align}

We obtained \eqref{eq:minus} by adding \eqref{eq:log-magic} (for $x = |X|$) to \eqref{eq:xbound}. The step from \eqref{eq:minus} to \eqref{eq:reduce}
follows by replacing the first occurrence of $|X|$ in \eqref{eq:minus}
with the upper-bound given in \eqref{eq:xbound}. The last line is then obtained by breaking the first logarithm in \eqref{eq:expand} and simplifying.
Hence we have $|X|\leq c' k \log (k+1)$ for some constant $c'$ depending on $\sigma'$ (and thus which can be bounded from above by a function of $c$), as desired.

If $H$ is not connected, then we reduce to the connected case similarly as in the end of the proof of \autoref{thm:main}, as follows. 
Let $H'$ be a connected planar graph on the same vertex set as $H$ containing $H$ as subgraph. 
First, we run the above algorithm for $H'$-models in $G$, which outputs a packing $\mathcal{P'}$ of $k'$ $H'$-models in $G$ and a subset $X'$ of vertices of $G$ such that $G-X'$ has no $H'$-model, of size at most $c' k' \log (k'+1)$, where $c':=c'(H')$. 
Observe that $\mathcal{P'}$ readily gives a packing of $k'$ $H$-models in $G$, since $H \subseteq H'$. 
Next, we use a theorem of Robertson and Seymour~\cite[Theorem 8.8]{robertson1986graph} stating that for every graph $J$ of treewidth at most $w$, one can find a packing $\mathcal{Q}$ of $H$-models in $J$ and a subset $Y$ of vertices of $J$ such that $J-Y$ has no $H$-model, of size at most $(b|\mathcal{Q}| -1)(w+1)$, where $b$ is the number of components of $H$.  
We apply this result to the graph $J=G-X'$, which has treewidth at most $w=f_{\sref{th:gridminor}}(|H'|)$ by \autoref{th:gridminor}.  
Note that $f_{\sref{th:gridminor}}(|H'|)=f_{\sref{th:gridminor}}(|H|)$ is a constant depending only on $H$. 
In particular, an optimal tree decomposition of $G-X'$ can be found in linear time using an algorithm of Bodlaender~\cite{B96}. 
Given this tree decomposition of $G-X'$, one can check that the proof given in~\cite{robertson1986graph} can be turned into a polynomial-time algorithm that finds the packing $\mathcal{Q}$ and the set $Y$ in polynomial time.
Alternatively, the problem of finding a largest packing of vertex-disjoint $H$-models is expressible in Monadic Second Order Logic (see~\cite{GroheSurvey}), and thus can be solved in polynomial time on graphs of bounded treewidth using Courcelle's theorem~\cite{B90}. 
As the same is true for the problem of finding a minimum size subset of vertices meeting all $H$-models, it follows that the aforementioned packing $\mathcal{Q}$ and subset $Y$ of vertices can be computed in polynomial time.  
Given $\mathcal{Q}$ and $Y$, and seeing $\mathcal{P'}$ as a packing of $H$-models in $G$, we let $\mathcal{P}$ denote the largest of the two packings $\mathcal{P'}$ and $\mathcal{Q}$, and let $X:=X'\cup Y$. 
Letting $k:= |\mathcal{P}|$, we thus find a packing of $k$ $H$-models in $G$, and a subset $X$ of at most $c' k' \log (k'+1) + (b|\mathcal{Q}| -1)(w+1) \leq c'k \log(k+1) + (bk-1)(f_{\sref{th:gridminor}}(|H|)+1)= O(k \log k)$ vertices of $G$ such that $G-X$ has no $H$-model, as desired. 

Now, we turn to the proof of \autoref{thm:main_technical}. 
In order to turn this proof into an algorithm, we will start the proof with $\mathcal{D}=(\emptyset, \dots, \emptyset)$, instead of an optimal orchard $(m, \omega)$-packing (which could be difficult to compute). 
Then, each time we apply one of the lemmas about orchards, we have the extra possibility that the current $(m, \omega)$-packing $\mathcal{D}$ could be improved (which could not happen when $\mathcal{D}$ was optimal), either by adding a new orchard to $\mathcal{D}$ which is vertex-disjoint from the existing ones, or by replacing existing orchards with better ones.  
Specifically, this could happen when following the proofs of \autoref{lem:prop_optimal_packing} or \autoref{lem:optimal_packing_bounded_model} (seen as algorithms) with our non-optimal $(m, \omega)$-packing $\mathcal{D}$.
Instead of producing the outcome guaranteed by these lemmas when $\mathcal{D}$ is optimal, these proofs could stop and output instead an $(m, \omega)$-packing $\mathcal{D'}$ of higher grade than $\mathcal{D}$.  
If this happens, we replace $\mathcal{D}$ with $\mathcal{D'}$ and restart from the beginning. 
Note that the grade of an  $(m, \omega)$-packing is at most $2^mn$, where $n$ is the number of vertices of $G$.
Therefore,  there will be at most linearly many such improvements of $\mathcal{D}$ because $m$ is a constant.  
Eventually, no improvement of $\mathcal{D}$ will be found anymore, and then we can follow the proof of \autoref{thm:main_technical} as if $\mathcal{D}$ were optimal. 

Since we restart the algorithm at most linearly many times, it only remains to check that the different steps in the proof of \autoref{thm:main_technical} can be done in polynomial time. 
Our use of the lemmas from Sections~\ref{sec:orchards} and~\ref{sec:packing_orchards} about orchards can be implemented in polynomial time because we always apply them to one (or two) orchard(s) from $\mathcal{D}$, and all orchards in $\mathcal{D}$ have size bounded from above by a constant. 
Thus, most steps of these proofs can be realized efficiently simply by using brute force. 
The only exceptions are the computations of matchings (in \autoref{lem:separateorchard} and \autoref{lem:separateorchard_extended}) and of vertex-disjoint paths using Menger's theorem (in \autoref{lem:separateorchard_extended} and \autoref{lem:separ}), which can be done in polynomial time using standard algorithms. 
The remaining steps of the proof of \autoref{thm:main_technical} are easily implemented efficiently. 
The only step requiring a comment is the use of \autoref{thm:small-minors} to obtain an $H$-model of size at most $\sigma \log |G|$. However, this can be done in polynomial time as well, as explained in~\cite{Fiorini20121226}. 

\section{Proofs of the remaining corollaries}
\label{sec:procor}
We prove in this section the results stated in \autoref{sec:combcor}
after the approximation algorithm.
We start with the proof of \autoref{thm:linear}. Since it is similar to that of \autoref{thm:main}, we shortened the common parts.
\begin{proof}[Proof of \autoref{thm:linear}.]
  First suppose that $H$ is connected. 
  From the facts that $\mathcal{G}$ is proper and minor-closed, we respectively deduce that there is a graph $F$ in the complement of $\mathcal{G}$, and that the graphs in $\mathcal{G}$ are $F$-minor free.
  Let $\sigma$ be the constant of \autoref{thm:main_technical} for the parameters $H$, $p: = |F|$, and let $g$ denote the function $f_{\sref{thm:fjw13}}$ of \autoref{thm:fjw13} for the graph $H$.
  We prove the result for $c:=\sigma$.
  As in the proof of \autoref{thm:main}, we consider a graph $G \in \mathcal{G}$ such that $\tau_H(G)>c\nu_H(G)$, with ($\nu_H(G), |G|, \Vert G \Vert)$ lexicographically minimum.
  Let us apply \autoref{thm:main_technical} on $G$ with the aforementioned parameters.
  According to \autoref{thm:fjw13}, the outcome \ref{e:sep} of \autoref{thm:main_technical} does not hold.
  As $G$ is $F$-minor free ---~this is the difference with the proof of \autoref{thm:main}~---, the outcome \ref{e:shallow-clique} is not possible either.
  Therefore $G$ contains an $H$-model of size at most $\sigma = c$. By considering the graph obtained by deleting this model, we can conclude as in   the proof of \autoref{thm:main}.

  Now assume that $H$ is not connected. 
  Then we can reduce to the connected case using the result of Robertson and Seymour~\cite[Theorem 8.8]{robertson1986graph}, exactly as in the description of the approximation algorithm in Section~\ref{sec:apalg}. (The only difference here is that when we apply the proof for a connected planar graph $H'$ containing $H$ as a subgraph, we obtain a linear bounding function for $H'$-models.)  
\end{proof}

We now move to the proof of \autoref{cor:parity}.
We actually prove a stronger statement, which we describe now. 
Given $m\in \N$, we say that $G$ is a $(0\mod m)$-subdivision of $H$ if $G$ can be obtained from $H$ by subdividing edges, 
in such a way that {\em every} edge of $H$ is replaced by a path having~$(0 \mod m)$ edges. 
A graph is \emph{subcubic} if it has no vertex of degree more than~$3$.
Thomassen~\cite{Thomassen1988} proved the following result.

\begin{theorem}[Thomassen {\cite[Theorem~3.3]{Thomassen1988}}]\label{thm:subcub}
  For every planar subcubic graph $H$ and every $m \in \N$ there is a
  function $f$ such that, for every $k \in \N$ and every graph $G$,
  either $G$ contains $k$ vertex-disjoint $(0 \mod m)$-subdivisions of
  $H$ as subgraphs, or there is a subset $X$ of at most $f(k)$ vertices such that $G-X$ contains
  no such subgraph.
\end{theorem}

Actually, our discussion of \autoref{thm:paritep} in
\autoref{sec:combcor} also applies to \autoref{thm:subcub}. Namely, the
function $f$ that can be extracted from Thomassen's proof is doubly
exponential in $k$ (for fixed $m$ and $H$) and was later improved to $f(k) =
O(k \log^d k)$ for some $d$ by Chekuri and
Chuzhoy~\cite{chekuri2013large}.
We prove here that \autoref{thm:subcub} holds for $f(k) = O(k \log k)$.

\begin{theorem}\label{thm:subcubopt}
  For every planar subcubic graph $H$ and every positive integer $m$
  there is a constant $c:=c(m,H)$ such that, for every
  $k \in \N$ and every graph $G$, either $G$ contains $k$
  vertex-disjoint $(0 \mod m)$-subdivisions of $H$ as subgraphs, or there is a subset
  $X$ of at most $c k \log (k+1)$ vertices such that $G-X$ contains no such subgraph.
\end{theorem}

As noted in \autoref{sec:combcor}, this bound is optimal, up to the
value of~$c$. 
\autoref{cor:parity} is the special case of \autoref{thm:subcubopt} where $H$ consists of a unique vertex with a loop.\footnote{The reader might rightly object that only simple graphs were considered so far in the paper. 
While it is true that the proof of \autoref{thm:subcubopt} works even if $H$ is a planar subcubic multigraph, let us mention the following alternative way of deducing \autoref{cor:parity} without resorting to multigraphs: Take $H$ to be a triangle. Then $(0 \mod m)$-subdivisions of $H$ correspond to cycles of lengths $0 \mod m$ that are at least $3m$. Thus, applying \autoref{thm:subcubopt} with this graph $H$, we obtain either $k$ vertex-disjoint such cycles in $G$, in which case we are done, or a subset $X$ of at most $c k \log (k+1)$ vertices such that $G-X$ has no such cycles. 
In the latter case, $G-X$ could still have cycles of length $m$ or $2m$. 
However, it suffices to take an inclusion-wise maximal packing of these in $G-X$: If we find at least $k$ of them, we are done. And if not, then we let $Y$ be the set at most $2m(k-1)$ vertices of the cycles in the packing. Then $G-(X\cup Y)$ has no cycle of length $0\mod m$, and $|X\cup Y| \leq c k \log (k+1) + 2m(k-1) = O(k \log k)$, as desired.}

Given a graph $J$ and a $J$-model $\{M_x: x\in V(J)\}$ in a graph $G$, let us define the {\em graph 
of the model} as the induced subgraph $G[\bigcup_{x\in V(J)} V(M_x)]$. 
We use the following lemma proved by Thomassen in his proof of~\autoref{thm:subcub}. 

\begin{lemma}[Thomassen {\cite[Proposition 3.2]{Thomassen1988}}]
\label{lem:thom} For every planar subcubic graph $H$ and every $m\in
\N$, there is a planar graph $J$ such that, for every graph $G$ 
and every $J$-model in $G$, the graph of the model contains a $(0 \mod m)$-subdivision of~$H$.
\end{lemma}

We will also need the following more precise version of Lemma~\ref{lem:minorep}, which appears in~\cite{ErdosPosaWheels}. 

\begin{lemma}[Aboulker, Fiorini, Huynh, Joret, Raymond, and Sau
\cite{ErdosPosaWheels}, reworded]\label{lem:minorep_alt} Let $J$ be a
planar graph and let $g$ be a bounding function for $J$-models.  If
$\mathcal{H}$ is a class of graphs such that for every $J$-model, the graph of the model contains a graph in $\mathcal{H}$ as a subgraph, then there is a function $f = O(g)$ such that, for every graph $G$ and every $k\in \N$, at least one of the following holds. 
  \begin{itemize}
  \item $G$ has $k$ vertex-disjoint subgraphs, each isomorphic to an
element of $\mathcal{H}$; 
  \item there is a set $X$ of at most $f(k)$ vertices such that $G-X$
has no subgraph isomorphic to an element of $\mathcal{H}$.
  \end{itemize}
\end{lemma}

The proof of \autoref{thm:subcubopt} is now immediate.
\begin{proof}[Proof of \autoref{thm:subcubopt}.]
Let $\mathcal{H}$ be the set of $(0 \mod m)$-subdivisions of $H$ and let $J$ be a planar
graph as in \autoref{lem:thom}.  By \autoref{thm:main}, there is a bounding function $g(k) = O(k \log k)$ for $J$-models.
Then using \autoref{lem:minorep_alt} with the graph $J$, the bounding function $g$, and the class $\mathcal{H}$, we obtain a bounding function of order $O(k \log k)$ satisfying the desired properties.
\end{proof}

Note that the same proof as above using a (linear) bounding function
provided by \autoref{thm:linear} instead of \autoref{thm:main} yields
the following corollary.
\begin{corollary}\label{cor:linear-parity}
  Let $\mathcal{G}$ be a proper minor-closed class, let $H$ be a
  planar subcubic graph, and let $m$ be a positive integer. Then there
  is a constant $c:=c(\mathcal{G}, H, m)$ such that, for every $k \in \N$
  and every graph $G\in \mathcal{G}$, either $G$ contains $k$
  vertex-disjoint $(0 \mod m)$-subdivisions of $H$ as subgraphs, or there is a subset
  $X$ of at most $c\cdot k$ vertices such that $G-X$ has no such subdivision.
\end{corollary}

We now address the proof of our partitioning corollary.
\begin{proof}[Proof of \autoref{cor:tw}.]
  For every $r \in \N$, let $\Gamma_r$ denote the $r\times r$-grid, which is known to have treewidth~$r$, and let $c_r$ be the constant $c$ in \autoref{thm:main} for $H := \Gamma_r$.
  We prove the result for $f(r) := (c_r + f_{\sref{th:gridminor}}(|\Gamma_r|))$. Let us consider a graph $G$ of treewidth at least~$f(r)\cdot k \log (k+1)$.
  If $\nu_{\Gamma_r}(G) \geq k$, then we are done.
  In the opposite case, by \autoref{thm:main}, $G$ has a set $X$ of at most $c_r \cdot k \log (k+1)$ vertices such that $G - X$ has no $\Gamma_r$-model.
  According to \autoref{th:gridminor}, $G-X$ has treewidth less than $f_{\sref{th:gridminor}}(|\Gamma_r|)$.
  By the properties of treewidth, we have $\tw(G) \leq \tw(G-X) + |X| < f_{\sref{th:gridminor}}(|\Gamma_r|) + c_r \cdot k \log (k+1)$.
  This is in contradiction with the assumption $\tw(G) \geq f(r)\cdot k \log (k+1)$.
\end{proof}

Let us conclude this section with the algorithms to compute minor-closed
bidimensional parameters. The proofs are essentially the same as
in~\cite{chekuri2013large} but since they are short, we include them
for completeness.

\begin{proof}[Proof of~\autoref{cor:miclopar1}]
  Let $k' = s_{\sref{cor:tw}}(p) (k+1) \log(k+2)$.
  We use an approximation algorithm for treewidth, for instance that
  of \cite{amir2010approximation}, that given $G$ and $k'$, either
  produces a tree-decomposition of $G$ of width $4k'$
  or correctly concludes that $\tw(G) \geq k'$, in  $2^{O(k')} n^{O(1)}$ time.

  In the first case, we use the algorithm required by the statement of the
  corollary on the tree-decomposition output by the approximation
  algorithm in order to decide whether $\pi(G) \leq k$ in time
  $h(4k') \cdot n^{O(1)}$.
  In the second case we immediately conclude that $(G,k)$ is a
  negative instance.
  Indeed, by \autoref{cor:tw}, $G$ then contains as a minor
  (actually, as subgraph) the disjoint union of $k+1$ graphs of
  treewidth at least $p$. From the properties of $\pi$ we deduce
  $\pi(G) \geq k+1$.
  The total worst-case running time is \[2^{O( s_{\sref{cor:tw}}(p)
    (k+1) \log(k+2))} n^{O(1)} + h(4 s_{\sref{cor:tw}}(p) (k+1)
  \log(k+2)) \cdot n^{O(1)},\] as claimed.
\end{proof}

\begin{proof}[Proof of~\autoref{cor:miclopar2}]
  Observe that $\pi$ is positive on all graphs with treewidth at
  least~$f_{\sref{th:gridminor}}(t)$, as any such graph contains $H$
  as a minor (\autoref{th:gridminor}) and $\pi$ is minor-closed.
  The result then follows from \autoref{cor:miclopar1} with~$p =
  f_{\sref{th:gridminor}}(t)$.
\end{proof}

\section*{Acknowledgements} 
\label{sec:ack}
\addcontentsline{toc}{section}{\nameref{sec:ack}}

We are much grateful to the three anonymous referees for their careful reading of the paper and their very helpful comments. 

\bibliographystyle{amsplain}


\begin{aicauthors}
\begin{authorinfo}[wcvb]
  Wouter Cames van Batenburg\\
  D\'epartement d'Informatique\\
  Universit\'e libre de Bruxelles\\
  Brussels, Belgium\\
  wcamesva\imageat{}ulb\imagedot{}ac\imagedot{}be\\
  \url{http://homepages.ulb.ac.be/~wcamesva}
\end{authorinfo}
\begin{authorinfo}[th]
  Tony Huynh\\
  D\'epartement de Math\'ematique\\
  Universit\'e libre de Bruxelles\\
  Brussels, Belgium\\
  tony\imagedot{}bourbaki\imageat{}gmail\imagedot{}com\\
  \url{https://sites.google.com/site/matroidintersection}
\end{authorinfo}
\begin{authorinfo}[gj]
  Gwena\"el Joret\\
  D\'epartement d'Informatique\\
  Universit\'e libre de Bruxelles\\
  Brussels, Belgium\\
  gjoret\imageat{}ulb\imagedot{}ac\imagedot{}be\\
  \url{http://di.ulb.ac.be/algo/gjoret}
\end{authorinfo}
\begin{authorinfo}[jfr]
  Jean-Florent Raymond\\
  Logic and Semantics Research Group\\
  Technische Universit\"at Berlin\\
  Berlin, Germany\\
  raymond\imageat{}tu-berlin\imagedot{}de\\
  \url{https://www.user.tu-berlin.de/jraymond}
\end{authorinfo}
\end{aicauthors}

\end{document}